\documentclass[%
leqno, 
DIV=10, 
english, 
titlepage=false, 
fontsize=11pt,
]{scrartcl}
\usepackage[utf8]{inputenc}
\usepackage[T1]{fontenc}
\usepackage{babel}
\usepackage{lmodern}
\usepackage{csquotes}
\usepackage{enumitem}
\usepackage{amsmath}
\usepackage{amssymb}
\usepackage{amsthm}
\usepackage{amsfonts}
\usepackage{mathtools}
\usepackage{esint}
\usepackage{cite}
\usepackage{array}
\usepackage{authblk}
\usepackage{nccmath}
\usepackage{bold-extra}

\usepackage{color}
\definecolor{LinkColor}{rgb}{0,0,1}
\definecolor{LinkColor2}{rgb}{0,0.5,0}
\definecolor{lg}{rgb}{.5,.5,.5}
\newcommand{\bk}{\color{black}}

\usepackage[%
left=1.25in,
right=1.25in,
bottom=1.5in,
top=1in,
footskip=0.5in,
]{geometry}

\usepackage[%
pdftitle={Titel},%
pdfauthor={Autor},%
pdfcreator={LaTeX, LaTeX with hyperref and KOMA-Script},
pdfsubject={Betreff}, 
pdfkeywords={Keywords}
]{hyperref} 
\hypersetup{%
	colorlinks	=true,
	linkcolor	=LinkColor,%
	citecolor	=LinkColor2,%
	urlcolor	=LinkColor,%
}

\setkomafont{section}{\centering\Large\rmfamily\normalfont\scshape}
\setkomafont{subsection}{\large\rmfamily\normalfont\bfseries}
\rmfamily

\numberwithin{equation}{section}

\usepackage{tocloft}

\newtheorem*{Abs}{Abstract}
\newtheorem{Thm}{Theorem}[section]
\newtheorem{Lem}[Thm]{Lemma}

\newtheorem{Cor}[Thm]{Corollary}
\newtheorem{Def}[Thm]{Definition}
\newtheorem{Rem}[Thm]{Remark}

\makeatletter
\renewenvironment{proof}[1][\proofname]{%
	\par\pushQED{\qed}\normalfont%
	\topsep6\p@\@plus6\p@\relax
	\trivlist\item[\hskip\labelsep\bfseries#1\@addpunct{.}]%
	\ignorespaces
}{%
	\popQED\endtrivlist\@endpefalse
}
\makeatother

\makeatletter
\renewcommand\paragraph{\@startsection{paragraph}{4}{\z@}%
	{1ex \@plus1ex \@minus.2ex}%
	{-1em}%
	{\normalfont\normalsize\bfseries}}
\renewcommand\subparagraph{\@startsection{paragraph}{4}{\z@}%
	{1ex \@plus1ex \@minus.2ex}%
	{-1em}%
	{\normalfont\normalsize\itshape}}
\makeatother

\newcommand{\bphi}{{\B\varphi}}
\newcommand{\bet}{{\B\eta}}
\newcommand{\btheta}{{\B\vartheta}}
\newcommand{\bx}{{\B x}}
\newcommand{\bw}{{\B w}}
\newcommand{\bu}{{\B u}}

\newcommand{\R}{\mathbb{R}}
\newcommand{\N}{\mathbb{N}}
\newcommand{\E}{\mathcal{E}}
\newcommand{\C}{\mathbb{C}}
\newcommand{\eps}{\varepsilon}

\newcommand{\bnorm}[1]{\big\| #1 \big\|}
\newcommand{\bBEC}[3]{\big\langle\mathcal{E}\big(#1\big),\mathcal{E}\big(#2\big)\big\rangle_{#3}}
\newcommand{\bBE}[2]{\big\langle\mathcal{E}\big(#1\big),\mathcal{E}\big(#2\big)\big\rangle_{\C(\bphi)}}

\newcommand{\brsp}[3]{\big(#1,#2\big)_{\rho(\B{#3})}}
\newcommand{\brsq}[3]{\big(#1,#2\big)_{\rho(#3)}}
\newcommand{\bBET}[3]{\big\langle \mathcal{E}\big(#1\big),\mathcal{E}\big(#2\big)\big\rangle_{\C({#3})}}
\newcommand{\bBK}[3]{\big\langle \mathcal{E}\big(#1\big),\mathcal{E}\left(#2\right)\big\rangle_{\C(\B{\varphi}_{#3})}}

\newcommand{\bDP}[2]{\big\langle#1,#2\big\rangle_{H^{-1},H^1} }
\newcommand{\HI}{H^1(\Omega;\R^d)}
\newcommand{\bLP}[2]{\big(#1,#2\big)_{L^{-1}} }
\newcommand{\bh}{\B{h}}

\numberwithin{equation}{section}  
 
\newcommand{\abs}[1]{\left\vert#1\right\vert}

\newcommand{\B}[1]{\boldsymbol{#1}}
\newcommand{\BE}[2]{\left\langle\mathcal{E}\left(#1\right),\mathcal{E}\left(#2\right)\right\rangle_{\C(\B{\varphi})}}
\newcommand{\BED}[2]{\left\langle \mathcal{E}(#1),\mathcal{E}(#2)\right\rangle_{\C^{\prime}(\B{\varphi})\B{h}}}
\newcommand{\BEDd}[2]{\left\langle \mathcal{E}(#1),\mathcal{E}(#2)\right\rangle_{\C^{\prime}(\B{\varphi})t\B{h}}}
\newcommand{\BET}[3]{\left\langle \mathcal{E}\left(#1\right),\mathcal{E}\left(#2\right)\right\rangle_{\C({#3})}}
\newcommand{\BEDp}[2]{\left\langle \mathcal{E}(#1),\mathcal{E}(#2)\right\rangle_{\C^{\prime}(\B{\varphi})(\tilde{\B{\varphi}}-\B{\varphi})}}

\newcommand{\DB}[2]{\boldsymbol{#1}^{\boldsymbol{#2}}}
\newcommand{\DP}[2]{\left\langle#1,#2\right\rangle_{H^{-1},H^1} }

\newcommand{\ev}{\lambda^{\B{\varphi}}_1}
\newcommand{\evd}{\left(\lambda_1^{\B{\varphi}}\right)^{\prime}\B{h}}

\newcommand{\HC}{H^1_C(\Omega;\R^d)}
\newcommand{\HD}{H^{-1}(\Omega;\R^d)}
\newcommand{\Hd}{H^1_D(\Omega;\R^d)}

\newcommand{\HL}{H^1(\Omega;\R^N)\cap L^{\infty}(\Omega;\R^N)}
\newcommand{\HN}{H^1(\Omega;\R^N)}
\newcommand{\ld}[2]{\left(\lambda^{\B{#1}}_{#2}\right)^{\prime}}
\newcommand{\LP}[2]{\left(#1,#2\right)_{L^{-1}} }
\newcommand{\LuN}{L^{\infty}(\Omega;\R^N)}
\newcommand{\Lzd}{L^2(\Omega;\R^d)}

\newcommand{\Lzrp}{L^2_{\B{\varphi}}(\Omega;\R^d)}	
\newcommand{\Lzrpa}[1]{L^2_{#1}(\Omega;\R^d)}

\newcommand{\norm}[1]{\left\Vert #1\right\Vert}
\newcommand{\oB}[1]{\overline{\B{#1}}}

\newcommand{\rsp}[3]{\left(#1,#2\right)_{\rho(\B{#3})}}
\newcommand{\rspd}[2]{\left(#1,#2\right)_{\rho^{\prime}(\B{\varphi})\B{h}}}
\newcommand{\rspt}[2]{\left(#1,#2\right)_{\rho^{\prime}(\B{\varphi})t\B{h}}}
\newcommand{\rsq}[3]{\left(#1,#2\right)_{\rho(#3)}}

\begin{document} 
	
\begin{titlepage}
	\begin{addmargin}{0.5 cm}
		\centering
		\LARGE{\scshape Shape and Topology Optimization \\ involving the eigenvalues\\ of an elastic structure:\\ A multi-phase-field approach}\\
		\rmfamily\mdseries
		\vspace{0.04\paperheight} 
		\normalsize
		\textsc{Harald Garcke, Paul H\"uttl and Patrik Knopf}\\[1ex]
		\textit{Fakul\"at f\"ur Mathematik, Universit\"at Regensburg, 93053 Regensburg, Germany}\\[1ex]
		\href{mailto:harald.garcke@ur.de}{harald.garcke@ur.de},
		\href{mailto:paul.huettl@ur.de}{paul.huettl@ur.de},
		\href{mailto:patrik.knopf@ur.de}{patrik.knopf@ur.de}\\
		\vspace{0.02\paperheight} 
		\begin{center}
		\small
		{\color{white}
			\textit{This is a preprint version of the paper. Please cite as:} \\  
			H. Garcke, P. H\"uttl and P. Knopf, [Journal] (2020) \\ 
			\texttt{https://doi.org/...}
		}
		\end{center}
		\vspace{0.02\paperheight}
		\small
		\begin{Abs}
			\normalfont
				A cost functional involving the eigenvalues of an elastic structure, that is described by a multi-phase-field equation, is optimized. This allows us to handle topology changes and multiple materials. We prove continuity and differentiability of the eigenvalues and we establish the existence of a global minimizer to our optimization problem. We further derive first order necessary optimality conditions for local minimizers. Moreover, an optimization problem combining eigenvalue and compliance optimization is also discussed. \\[1ex]
		\end{Abs}
		\bk\\
	\end{addmargin}
\end{titlepage}


\bigskip
\normalsize
\setlength{\parskip}{1ex}
\setlength{\parindent}{0ex}


\section{Introduction} 
The main goal in structural topology optimization is to find the optimal distribution of materials in a so called \emph{design domain}. 
In contrast to shape optimization, topological changes including the design and distribution of holes in the structure are also allowed in topology optimization. The shape and the topology of the structure are initially unknown and are to be determined by minimizing a suitable objective functional.
In many applications (e.g., design engineering), support conditions, volume restrictions, prescribed solid regions or voids, as well as applied loads have to be taken into account.

Several mathematical techniques to deal with shape or topology optimization problems can be found in the literature. The traditional approach is the method of boundary variations to compute shape derivatives. In this way the value of the objective functional can be decreased by deforming the boundary in a certain descent direction (see, e.g., \cite{Delfour,Murat-Simon,Simon,Sokolowski} and the references cited therein). The drawbacks of this technique are its high computational costs and that topological changes are not allowed. In some situations, it is also possible to deal with changes of the topology by homogenization methods (see, e.g., \cite{AllaireBook}) or variants of this approach such as the SIMP method (see, e.g., \cite{Bourdin}). Especially in recent times, the level-set method has been a popular tool to approach topology optimization problems. It was originally developed in \cite{Osher-Sethian} and was afterwards frequently used in the literature (see, e.g., \cite{Burger,Osher-Santosa}). Although this method can handle topological changes, difficulties can arise if voids are to be created.

In this paper, however, we pursue a different ansatz. We describe an elastic structure by a vector valued phase-field variable $\B{\varphi}$ representing the distribution of materials. 
Here, with respect to the space variable the phase-field $\B{\varphi}$ does not change its values abruptly but exhibits continuous phase transitions. 
The phase-field method for topology optimization was first introduced in \cite{Bourdin} and was subsequently used, e.g., in \cite{Blank,Blank2,Blank3,Burger-Stainko,Dede,Penzler,Takezawa,Zhou-Wang,Zhou-Wang2,Zhou-Wang3}.
The main advantage of this approach is that topological changes can be handled directly without having to switch the framework. In particular, the creation of voids does not impose any problems. Moreover, these models are very well suited to be treated with methods from mathematical analysis.

In many technical applications (especially from engineering sciences), it is not only desired to optimize the material distribution of an elastic structure but also to minimize (or maximize) its dynamical
response to a given driving frequency of a specific range.  
A classical example is a machine whose running engine generates vibrations that affect other parts of the device. Mostly, it is essential that the vibrations do not match certain \emph{eigenfrequencies} of other components of the machine to avoid a resonance disaster.

A typical concrete example is discussed in \cite{Bendsoe} where the engine of an airplane is considered. While the engine is running it creates vibrations that affect other components of the aircraft, especially its wings. Therefore, they must be designed in such a way that these vibrations are not amplified. Otherwise this could lead to fatal flight instabilities and might even damage or break the material.
From a mathematical point of view, the \emph{eigenmodes} of the engine and the wings should be taken into account within the optimization process to keep them as different from the modes of the engine as possible. One possibility to guarantee this behavior would be to maximize the smallest eigenmode of the wings, as consequently also all larger eigenmodes will be separated from those of the engine (which are generally rather small). 

Models for eigenvalue problems and their analysis have already been discussed in the literature, see \cite{Allaire2, Barbarosie, Bucur, Buttazzo, Elliot, Henrot, Rousselet}.
In \cite{Barbarosie, Allaire2} the authors investigate models similar to the one we intend to study. In these papers, the density distribution $\rho$ is assumed to depend only on the spatial variable $\bx\in\Omega$ meaning that the dependence on the structure (represented by the phase-field $\bphi$) is neglected. However, as the material distribution of the structure is actually to be optimized, the optimal density distribution is initially unknown.

In this paper, we study the following eigenvalue problem to describe an elastic structure:
\begin{alignat}{3}
    \label{EWP:0}
        \left\{
            \begin{aligned}
                -\nabla \cdot\left[\mathbb{C}(\B{\varphi})\mathcal{E}(\B{w})\right]
            &=
                \lambda^{\B{\varphi}}\rho(\B{\varphi})\B{w}&&\text{in }\Omega,\\
                \B{w}
            &=
                \B{0}&&\text{on }\Gamma_D,\\
                \left[\mathbb{C}(\B{\varphi})\mathcal{E}(\B{w})\right]\B{n}
            &=
                \B{0} &&\text{on }\Gamma_0,
            \end{aligned}
    \right.
\end{alignat}
with the disjoint splitting $\partial\Omega=\Gamma_D\cup \Gamma_0$, where we additionally demand $\Gamma_D$ to have strictly positive Hausdorff measure. Here, $\mathbb{C}$ denotes the elasticity tensor, $\mathcal E(\bw)$ stands for the symmetrized gradient of $\bw$, $\lambda^\bphi$ is the eigenvalue (depending on $\bphi$) and $\B w= \B w^\bphi$ denotes a corresponding eigenfunction. In contrast, to the similar models studied in \cite{Barbarosie, Allaire2}, the density distribution $\rho(\bphi)$ is now allowed to depend on the phase-field $\bphi$. For more details about the notation and the motivation of this model, we refer the reader to Section~2.

We prove the existence of eigenvalues and eigenfunctions for the problem \eqref{EWP:0} and we establish essential properties needed for the theory of calculus of variations such as suitable continuity statements. This allows us to investigate an optimal control problem where an objective functional 
\begin{align}\label{objFunc}
        J_{l}^{\varepsilon}(\B{\varphi})
    =
            \Psi(\lambda_{i_1}^{\B{\varphi}},\dots, \lambda_{i_l}^{\B{\varphi}})
        +
            \gamma E^{\varepsilon}(\B{\varphi}),
\end{align}
is to be minimized under the constraint that $\B{\varphi}$ and $\lambda_{i_j}^\bphi$ satisfy the state equation \eqref{EWP:0}. Here, the function $\Psi:(\mathbb R_{>0})^l\to\mathbb R$ is continuously differentiable and bounded from below, and penalizes the eigenvalues. The expression $E^{\varepsilon}(\B{\varphi})$ stands for the 
\emph{Ginzburg--Landau energy}
\begin{align}
        E^{\varepsilon}(\B{\varphi})
    =
        \int_{\Omega}
            \left(
                \frac{\varepsilon}{2}\abs{\nabla\bphi}^2+\frac{1}{\varepsilon}\psi(\B{\varphi})
            \right), 
\end{align} 
where $\varepsilon>0$ corresponds to the thickness of the diffuse interface and $\psi$ stands for the bulk potential that usually has a double-well structure. A typical choice is $\psi(s) = (s^2-1)^2$, $s\in\mathbb R$.
As the energy $E^{\varepsilon}$ is an approximation of the perimeter of the material boundaries, minimizing \eqref{objFunc} can be related to a shape and topology optimization problem with a perimeter penalization (see, e.g., \cite{Ambrosio}). The phase-field $\bphi$ represents the control and is supposed to satisfy suitable restrictions. For reasons of mathematical analysis we use a diffuse interface approach, i.e. the components of $\B{\varphi}$ do not change their values abruptly but continuously at interfacial regions between the materials. As in \cite{Blank}, the sharp interface limit could be considered to describe a discrete material distribution and would allow it to become a formulation with a perimeter regularization.
The optimization problem will be introduced in more detail in Section~2.

To derive first-order necessary conditions for locally optimal controls we need to show that the considered eigenvalues $\lambda^\bphi$ are differentiable with respect to $\bphi$. 
Based on the theory developed in \cite{Rousselet}, we can show that the smallest eigenvalue is semi-differentiable with respect to the phase-field. In addition, we prove Fr\'echet differentiability of simple eigenvalues by means of the implicit function theorem after introducing a proper sign convention for the eigenfunctions. A positive side benefit of the implicit function theorem is that we also obtain the Fr\'echet derivatives of the corresponding eigenfunctions.

With the approach described in this paper, it is also possible to approach classical shape and topology optimization problems also for other elliptic operators, i.e., the Laplacian. The idea is to send $\mathbb{C}$ and $\rho$ to $0$ in the void phase. In a forthcoming paper we plan to study this limit in more detail.
To describe the idea of our approach in such a setting we consider a phase-field approximation for a spectral optimization problem for the Neumann--Laplace operator.
Here, a domain $D\subset \mathbb{R}^n$ is to be optimized such that the eigenvalues of the Neumann problem for the Laplace operator
\begin{alignat*}{2}
    -\Delta u&=\lambda u&&\quad \text{in }D,\\
    \nabla u\cdot \B{n}&=\B{0}&&\quad \text{on }\partial D,
\end{alignat*}
are maximal. Choosing a scalar phase-field variable $\varphi$ and a function $a$ such that $a(-1)=\delta$, where $\delta$ denotes a small parameter, $a(1)=1$ such that $a$ is a smooth, positive interpolation in between. We then solve
\begin{alignat*}{2}
        -\nabla\cdot\left(a(\varphi)\nabla u\right)
    &=
        \lambda a(\varphi) u&&\quad \text{in }D,\\
        \nabla u\cdot \B{n}
    &=
        \B{0}&&\quad \text{on }\partial D,
\end{alignat*}
to obtain eigenvalues $0=\lambda_{1}^{\varphi}<\lambda_2^{\varphi}\le \lambda_3^{\varphi}\le\dots$. We then optimize
\begin{align*}
    \Psi\left(\lambda_{i_1}^{\varphi},\dots,\lambda_{i_l}^{\varphi}\right)
    +\gamma\int_{\Omega}
        \left(
            \frac{\varepsilon}{2}\abs{\nabla\varphi}^2
            +\frac{1}{\varepsilon}\psi(\varphi)
         \right),  
\end{align*}
with a given function $\Psi$. We conjecture that for $\varepsilon,\delta\to 0$ classical spectral optimization problems for the Laplace operator are recovered. As
\begin{math}
      \int_{\Omega}
         \left(
            \frac{\varepsilon}{2}\abs{\nabla\varphi}^2
            +\frac{1}{\varepsilon}\psi(\varphi)
          \right) 
\end{math}
converges in the sense of a $\Gamma$-limit to the perimeter functional the limits contain a perimeter regularization. In a similar way \eqref{objFunc} under the constraint \eqref{EWP:0} can be related to a sharp interface problem for eigenvalues of the elasticity operator.

Our paper is structured as follows. First, we precisely formulate the mathematical model for the problem with a special emphasis on the eigenvalue problem and its analytic difficulties. After the first continuity results for eigenvalues and eigenfunctions with respect to the phase-field $\B{\varphi}$, we are able to show existence of a minimizer of the objective functional. Here we do not need to assume anything about simplicity of eigenvalues.\\
The most elaborate part is then dedicated to deriving differentiability results and to improve the aforementioned continuity statements, which will yield the desired variational inequality. In this context, it is crucial to assume simplicity of the considered eigenvalues.
In the last part we want to combine the eigenvalue problem with compliance minimization problems. 


\section{Formulation of the problem} 
This section is devoted to the introduction of the mathematical model and the structural optimization problem.

\subsection{The design domain and the phase-field variable}
We fix a bounded Lipschitz design domain $\Omega\subset \R^{d}$ with $d\in \N$ whose boundary is split into two disjoint parts: A homogeneous Dirichlet boundary $\Gamma_D$ with strictly positive $(d-1)$-dimensional Hausdorff measure and a homogeneous Neumann boundary $\Gamma_0$. The distribution of $N\in \N$ materials is represented by the vector valued phase-field $\bphi: \Omega \to \R^N$. 
This means, for any $i\in\{1,...,N\}$, the component $\bphi_i$ can be interpreted as the concentration of the $i$-th material.
In this regard, $\bphi_i = 0$ describes the absence of the $i$-th material, whereas $\bphi_i = 1$ means that only the $i$-th material is present.
We use the convention that voids are also interpreted as a sort of material, whose distribution is given by the $N$-th component of the vector $\bphi$.  

For reasons of mathematical analysis we use a diffuse interface approach, i.e. the components of $\bphi$ do not change their values abruptly but continuously at interfacial regions between the materials.

Furthermore, we want to prescribe the total spatial amount of each phase. To this end, we impose the mean value constraint $\fint_{\Omega}\bphi=\B{m}=(m^{i})_{i=1}^{N}$, 
where $m^{i}\in (0,1)$ is a fixed given number for any $i\in \left\{1,\dots , N\right\}$. 
In addition, we want the vector $\B{m}$ to be an element of the set
\begin{align*}
    \Sigma^N
    =\left\{
            \B{v}\in \R^N\left| \,\sum_{i=1}^{N}v^{i}
            =1
        \right.\right\}. 
\end{align*}
This constraint is a plausible consequence of the physical fact that at each point in space the volume fractions of the materials should sum up to $1$. Furthermore, being a volume fraction, each component clearly has to be non-negative. For the upcoming analysis, we additionally want to prescribe a suitable regularity for the phase-field, namely $\HN$.
All these constraints are expressed in the set
\begin{align*}
    \B{\mathcal{G}}^{\B{m}}=\left\{\B{v}\in \B{\mathcal{G}}\left|\,\fint_{\Omega}\B{v}=\B{m}\right.\right\}.
\end{align*}
Here, $\B{\mathcal{G}}$ is given by
\begin{align*}
        \B{\mathcal{G}}
    =
        \left\{
            \left.\B{v}\in \HN \right|\, \B{v}(\bx)\in \B{G}\;\text{for almost all}\; \bx\in\Omega
        \right\},
\end{align*}
where $\B{G}=\R^N_{+}\cap \Sigma^N$ with
\begin{align*}
        \R^N_{+}
    =
        \left\{
            \left.\B{v}\in \R^N\right| v^{i}\ge 0 
            \quad\forall i\in \left\{1,\dots, N\right\}
        \right\}.
\end{align*}
The set $\B{G}$ is referred to as the \emph{Gibbs-Simplex}.

\subsection{The Ginzburg--Landau energy}

For the objective functional and especially the well-posedness of the minimization problem the following so called \emph{Ginzburg--Landau energy}
\begin{align*}
        E^{\eps}(\bphi)
    =
        \int_{\Omega}
            \left(
                \frac{\eps}{2}\abs{\nabla\bphi}^2+\frac{1}{\eps}\psi(\bphi)
            \right), 
        \quad \eps>0,
\end{align*}
is crucial. In our model the function $\psi: \R^N\to \R\cup \left\{\infty \right\}$ should attain exactly $N$ global minima of value zero attained at the unit vectors $\B{e}_i\in \R^N$, i.e. for all $i\in\{1,...,N\}$,
\begin{align*}
    \min\psi = \psi(\B{e}_i) = 0.
\end{align*}
Furthermore, $\psi$ is assumed to exhibit the decomposition $\psi(\bphi)=\psi_0(\bphi)+I_{\B{G}}(\bphi)$ 
with $\psi_0\in C^{1,1}(\R^N,\R)$ 
and $I_{\B{G}}$ being the indicator functional
\begin{align*}
        I_{\B{G}}(\bphi)
     =
        \begin{cases}
            0&\text{if } \bphi\in\B{G},\\
            \infty& \text{otherwise}.
        \end{cases}
\end{align*}
This type of obstacle functional is used to enforce that $\bphi$ attains its values only in $\B{\mathcal{G}}$ as this set is not penalized by the indicator functional. We refer to Elliot and Luckhaus \cite{Elliott2} who first introduced this obstacle formulation of the energy $E^{\varepsilon}$.

\subsection{The density function}

The density distribution $\rho$ depends directly on the phase-field $\bphi$ and this way, $\rho$ is not just a given function but represents the density of the actual structure we want to optimize.

To this end, we assume that the density function $\rho$ belongs to $C^{1,1}_\text{loc}(\R^N;\R)$ and is uniformly positive, i.e., there exists a constant $\rho_0>0$ such that $\rho(\bphi)\ge \rho_0$ for all $\bphi\in \R^N$. This directly yields
\begin{align}
    \rho_0\abs{\bu}^2&\le \rho(\bphi)\abs{\bu}^2
\end{align}
for all $\bphi, \B u\in \R^N$. For any fixed $\bphi\in \R^N$, there exist constants $ C_{\bphi},C^{\prime}_{\bphi}>0$ (that may locally depend on $\bphi$, i.e., $C_{\bphi}$ and $C^{\prime}_{\bphi}$ can be chosen uniformly on bounded sets), such that
\begin{alignat}{2}\label{rhabs}
    \begin{aligned}
            \abs{\rho(\bphi)\bu\cdot \B{v}}
        &\le 
            C_{\bphi}\abs{\bu}\abs{\B{v}},\\
            \abs{\rho^{\prime}(\bphi)\B{h}\bu\cdot\B{v}}
        &\le 
            C^{\prime}_{\bphi}\abs{\B{h}}\abs{\bu}\abs{\B{v}},
    \end{aligned}
\end{alignat}
for all $\bu,\B{v}\in \R^d$. Next, for any function $\varrho\in C(\R^N;\R)$, we define
\begin{align*}
        \big({\B{f}},{\B{g}}\big)_{\varrho}
     \coloneqq
        \int_{\Omega}\varrho\,\B{f}\cdot\B{g}\text{\,d}x
      \quad \text{for all}\; \B f, \B g\in \Lzd.
\end{align*}
Due to the above assumptions, we can use this notation to define a family of scalar products on $\Lzd$ depending on $\bphi\in \LuN$ by
\begin{align}\label{SKP}
        \rsp{\B{f}}{\B{g}}{\bphi}
    \coloneqq 
        \int_{\Omega}\rho(\bphi)\B{f}\cdot\B{g}\text{\,d}x
     \quad \text{for all}\; \B f, \B g\in \Lzd.
\end{align} 
These scalar products canonically induce norms that are all equivalent to the standard norm on $\Lzd$. To indicate the norm we consider $\Lzd$ to be equipped with, we will use the notation $\Lzrp$.

A reasonable choice of $\rho$ would be
\begin{align}\label{EX:RHO}
        \rho(\bphi)
    =
        \sum_{i=1}^{N} \varrho_i\, \varphi_i
    =
        \sum_{i=1}^{N-1}\varrho_i\, \varphi_i+\varepsilon^2\tilde{\varrho}_N\varphi_N, \quad \bphi\in \B G.
\end{align}
Here, for any $i\in\{1,...,N-1\}$, the coefficient $\varrho_i>0$ stands for the density of the $i$-th material which is assumed to be constant. In our model we interpret the void as a material of very low density.
Hence, we chose $\varrho_N=\varrho_N(\varepsilon)=\varepsilon^2\tilde{\varrho}_N$ as corresponding density, where $\tilde{\varrho}_N>0$ is a fixed constant. The scaling with $\varepsilon^2$ would guarantee the desired behaviour of $\rho$ in the sharp interface limit, see \cite{Blank} who treat the sharp interface limit for a related problem.

However, for the sake of mathematical analysis, we have to extend the definition of $\rho$ onto the hyperplane $\Sigma^N$. To this end, we define the cut-off function
\begin{align}
        \sigma_\delta:\R\to\R
        \quad s
    \mapsto
        \begin{cases}
            -\delta &\text{if}\;\; s \le -\delta, \\
            a_\delta  &\text{if}\;\; -\delta< s<0, \\
            s &\text{if}\;\; 0\le s \le 1,\\
            b_\delta &\text{if}\;\; 1 < s < 1+\delta, \\
            \delta &\text{if}\;\; s \ge 1+\delta, \\
        \end{cases}
\end{align}
for any $\delta>0$ which will be specified later.
Here, $a_\delta$ and $b_\delta$ are monotonically increasing $C^{1,1}$-functions such that $\sigma_\delta \in C^{1,1}(\R;\R)$. 
We now define the function $\rho$ by
\begin{align}
\label{DEF:RHO}
    \rho:\R^N\to\R,\quad  \bphi\mapsto 
    \sum_{i=1}^{N} \varrho_i\, \sigma_\delta(\varphi_i), 
    \quad \bphi\in \Sigma^N.
\end{align} 
Obviously, it holds that $\rho \in C^{1,1}(\R^N;\R)$ and
since $\sigma_\delta(\bphi_i) = \bphi_i$ for all $i\in\{1,...,N\}$ as long as $\bphi\in \B{G}$, the relation \eqref{EX:RHO} holds true for this definition.

It remains to show that $\rho$ is uniformly positive, at least if $\delta$ is chosen sufficiently small. To this end, we fix an arbitrary vector $\bphi\in\Sigma^N$ and define the index sets
\begin{align*}
        I := \{1,...,N\},
    \quad
        I_{<0}:=\{ i \in I \,\vert\, \varphi_i < 0 \},
    \quad
        I_{\ge 0} := I \setminus I_{<0}.
\end{align*}
Recalling the definition of $\Sigma^N$, we infer that
\begin{align*}
        \sum_{I_{\ge 0}} \varphi_i \ge 1,
    \quad\text{and thus also}\quad 
        \sum_{I_{\ge 0}} \sigma_\delta(\varphi_i) \ge 1.
\end{align*}
Choosing
\begin{align*}
        M:=\underset{i\in I}{\max}\; \varrho_i, 
    \quad
        m:=\underset{i\in I}{\min}\; \varrho_i,
    \quad
        \delta := \frac{m}{2MN} > 0,
    \quad\text{and}
    \quad 
        \rho_0 := \frac m 2 > 0 
\end{align*}
we conclude the estimate
\begin{align*}
        \rho(\bphi) 
    = 
        \sum_{I_{\ge 0}} \varrho_i\, \sigma_\delta(\bphi_i) \; 
     + \; \sum_{I_{< 0}} \varrho_i\, \sigma_\delta(\bphi_i) 
    \ge 
        m - \delta M N 
    = 
        \rho_0 > 0.
\end{align*}
Since $\bphi\in \Sigma^N$ was arbitrary, this estimate holds for all $\bphi\in \Sigma^N$. We point out that $\rho_0$ does not depend on $\bphi$ and thus, this estimate is uniform. This means that the function $\rho$ defined in \eqref{DEF:RHO} is admissible as it exhibits all demanded properties.

\subsection{The elasticity tensor}

Another important tool in linear elasticity are the tensors appearing in Hooke's Law (see, e.g.,\cite{Eck,Gurtin}), namely the strain and the elasticity tensor which describe the stress tensor.
To introduce the strain tensor we consider the displacement vector $\bu: \Omega \to \R^d$ 
that describes the deformation of the structure under applied forces or vibrations. Now, the strain tensor of $\bu$ can be defined as
\begin{align*}
    \mathcal{E}(\bu)\coloneqq\left(\nabla \bu\right)^{\text{sym}},
\end{align*}
where $\mathcal{A}^{\text{sym}}\coloneqq \frac{1}{2}\big(\mathcal{A}+\mathcal{A}^T\big)$ for any matrix $\mathcal{A}\in \R^{d\times d}$.
The elasticity tensor $\C$ is a fourth order tensor whose components are demanded to fulfill $\C_{ijkl}\in C^{1,1}_\text{loc}(\R^N,\R)$ as well as the symmetry properties
\begin{align}\label{sp}
    \C_{ijkl}=\C_{jikl}=\C_{ijlk}=\C_{klij}.
\end{align}
for all $i,j,k,l\in \left\{1,\dots, d\right\}$.
From the regularity property we conclude that for any $\bphi\in \R^N$, there exist constants $\Lambda_{\bphi},\Lambda^{\prime}_{\bphi}>0$ locally depending on $\B{\varphi}$ such that
\begin{alignat}{2}\label{PC1}
    \begin{aligned}
            \abs{\C(\bphi)\mathcal{A}:\mathcal{B}}
        &\le 
            \Lambda_{\bphi}\abs{\mathcal{A}}\abs{\mathcal{B}},\\
            \abs{\C^{\prime}(\bphi)\B{h}\mathcal{A}:\mathcal{B}}
        &\le 
            \Lambda^{\prime}_{\bphi}\abs{\B{h}}\abs{\mathcal{A}}\abs{\mathcal{B}}, 
    \end{aligned}
\end{alignat}
for all symmetric matrices 
$\mathcal{A},\mathcal{B}\in \R^{d\times d}\backslash\left\{ \B{0}\right\}$ and $\B{h}\in \R^N$,
where
\begin{align*}
        \mathcal{A}:\mathcal{B} 
    &\coloneqq 
        \sum_{i,j=1}^{d}\mathcal{A}_{ij}\mathcal{B}_{ij}\;,
\end{align*}
and
\begin{align*}
        \C^{\prime}(\bphi)\B{h}
    &= 
        \left(\sum_{m=1}^{N}\partial_m\C_{ijkl}(\bphi)h_m\right)_{i,j,k,l=1}^{d}.
\end{align*}
denotes the derivative of $\C(\bphi)$ in the direction $\B h$.
Furthermore, we demand that there exists a positive constant $\theta$ such that for all symmetric matrices $\mathcal{A}\in \R^{d\times d}\backslash\left\{ \B{0}\right\}$ 
and for all $\bphi,\B{h}\in \R^N$ it holds
\begin{align}\label{PC2}
    \theta \abs{\mathcal{A}}^2\le \C(\bphi)\mathcal{A}:\mathcal{A}.
\end{align}
Recall that the application of a fourth order tensor onto a quadratic matrix is given by
\begin{align*}
    \left(\C\mathcal{A}\right)_{ij}=\sum_{k,l=1}^{d}\C_{ijkl}\mathcal{A}_{kl}.
\end{align*}
A concrete choice of the elasticity tensor in analogy to the construction of $\rho$ is
\begin{align*}
        \mathbb{C}(\bphi)
    =
        \sum_{i=1}^{N}\mathbb{C}_i\varphi_i
    =
        \sum_{i=1}^{N-1}\mathbb{C}_i\varphi_i+\varepsilon^2\tilde{\mathbb{C}}_N\varphi_N,
        \quad \bphi\in \B G,
\end{align*}
where  for $i=1,\dots,N-1$, $\mathbb{C}_i$ and $\tilde{\mathbb{C}}_N$ denote constant material specific elasticity tensors. To guarantee \eqref{PC2} we need to assume the existence of positive constants $\tilde{\vartheta}_i,\vartheta_i$ such that for all ${\mathcal{A}}\in \mathbb{R}^{d\times d}\backslash \left\{\B{0}\right\}$, it holds that
\begin{align*}
        \tilde{\vartheta}_i\abs{{\mathcal{A}}}^2
    \le 
        \mathbb{C}^i{\mathcal{A}}:{\mathcal{A}}
    \le
        \vartheta_i\abs{{\mathcal{A}}}^2,
\end{align*}
for all $i=1,\dots,N$. Now, proceeding similarly as for the density $\rho$, we can construct an extension to $\mathbb{R}^N$ taking \eqref{PC2} into account. For more details we refer to \cite[Sect.~2.2]{Blank}. 

\subsection{The system of PDEs describing the elastic structure}

We now introduce the system of equations describing the elastic structure:
\begin{alignat}{3}
    \left\{
        \begin{aligned}\label{EWP}
                -\nabla \cdot\left[\C(\bphi)\mathcal{E}(\bw)\right]
            &=
                \lambda^{\bphi}\rho(\bphi)\bw&&\text{in }\Omega,\\
                \bw
            &=
                \B{0}&&\text{on }\Gamma_D,\\
                \left[\C(\bphi)\mathcal{E}(\bw)\right]\B{n}
            &=\B{0} &&\text{on }\Gamma_0.
        \end{aligned}
    \right.
\end{alignat}
Here, $\B{n}$ is the outer unit normal vector to the boundary $\partial\Omega=\overline{\Gamma_D\,\cup\,\Gamma_0}$. The subsets  $\Gamma_D,\Gamma_0\subset \partial\Omega$ are relatively open and satisfy $\Gamma_D\,\cap\,\Gamma_0 = \emptyset$ and $\mathcal{H}^{d-1}\left(\Gamma_D\right)>0$, where $\mathcal{H}^{d-1}$ denotes the $\left(d-1\right)$-dimensional Hausdorff measure.
To consider this problem in the weak sense, we define the closed subspace
\begin{align*}
        H^1_D(\Omega;\R^d) 
    := 
        \left\{ \left.\bet\in H^1(\Omega;\R^d) \,\right|\, \bet = \B 0 \;\text{a.e. on}\; \Gamma_D \right\}
    \subset 
        H^1(\Omega;\R^d).
\end{align*}
Endowed with the standard inner product and norm given by
\begin{align*}
        (\cdot{,}\cdot)_{\Hd} 
    := 
        (\cdot{,}\cdot)_{\HI}\,, 
        \quad \norm{\,\cdot\,}_{\Hd}
    := 
        \norm{\,\cdot\,}_{\HI}\,,
\end{align*}
$\Hd$ is a Hilbert space.
For any matrices $\mathcal{A},\mathcal{B}\in \R^{d\times d}$ and any fourth-order tensor $\mathcal C\in\R^{d\times d \times d\times d}$, we introduce the notation
\begin{align*}
        \langle \mathcal{A},\mathcal{B}\rangle_{\mathcal C}
    &\coloneqq 
        \int_{\Omega}\mathcal{A}:\mathcal C\mathcal{B}\text{\,d}x.
\end{align*}
Then the mapping
\begin{align}\label{eq:newinner}
    \BE{\cdot}{\cdot} : \Hd\times\Hd \to \R,\quad (\B w,\B \eta)\mapsto \BE{\bw}{\bet}
\end{align}
defines a scalar product on $\Hd$.
By Korn's inequality (see, e.g., \cite{Zeidler4}), the norm induced by this inner product is equivalent to the standard norm on $\Hd$. In what follows, we will always choose for a given $\B{\varphi}$ this inner product and induced norm on $\Hd$.

Using this notation and invoking the symmetry property \eqref{sp}, the weak formulation of $\eqref{EWP}$ can be expressed as
\begin{align}\label{WWP}
        \BE{\bw}{\bet}
    =
        \lambda^{\bphi}\rsp{\bw}{\bet}{\bphi}
        \quad\text{for all}\; \bet\in H^1_D(\Omega;\R^d).
\end{align}
In Section~3, we will see that for any $\bphi\in L^\infty(\Omega,\R^N)$, there exists a sequence of eigenvalues 
\begin{align*}
    0<\lambda_1^{\bphi}
    \le\lambda_2^{\bphi}
    \le\lambda_3^{\bphi}
    \le\cdots \to \infty
\end{align*}
and corresponding eigenfunctions $\{\bw_1^\bphi,\bw_2^\bphi,...\}\subset H^1_D(\Omega;\R^d)$ which form an orthonormal basis of $L^2_\bphi(\Omega;\R^d)$.

We also want to mention that in \cite{Rousselet} such eigenvalue problems are analyzed on a general, rather abstract level. The eigenvalue problem therin is given by the relation
\begin{align}
	\label{ROUS:1}
	a\big(\B\varphi;w(\B\varphi),\B\eta\big) = \lambda^{\B\varphi} b\big(\B\varphi;w(\B\varphi),\B\eta\big),\quad \B\eta\in H
\end{align}
where $\lambda^{\B\varphi}$ and $w(\B\varphi)$ stand for the corresponding eigenvalue and eigenfunction, respectively, $H$ is a Hilbert space and $\B{\eta}\in H$ can be interpreted as a test function. For the analysis, it is assumed that $a$ and $b$ are of the form
\begin{align}
	\label{ROUS:2}
	a(\B\varphi;w,\B\eta) = \big(A(\B\varphi) w(\B\varphi), \B\eta\big)_H\,,
	\quad
	b(\B\varphi;w,\B\eta) = \big(B(\B\varphi) w(\B\varphi), \B\eta\big)_H\,,	
\end{align}
where $A$ and $B$ are linear, continuous operators, $B$ is compact, and $(\cdot,\cdot)_H$ denotes the inner product on $V$.
For this problem, continuity and (semi-)differentiability of $\lambda^\bphi$ and $\B w(\bphi)$ with respect to $\bphi$ is established using an approach involving inverse operators that differs from the one discussed in this paper. However, optimization problems are not addressed in \cite{Rousselet}. The concept of ``semi-differentiability'' applied to our setting will be explained in Section~5.1 in more detail.

Next, we introduce the structural optimization problem in which 
the system \eqref{EWP} can be regarded as the state equation. 

\subsection{The structural optimization problem} \label{SOPT}

We also want to introduce constraints on the structure which prescribe void or material in certain parts of $\Omega$. Mathematically speaking, we fix two disjoint measurable sets $S_i\subset \Omega$ with $i\in \left\{0,1\right\}$ and define the set
\begin{align*}
        \B{U}_{c}
    \coloneqq 
        \left\{
            \bphi\in \HN
            \left| 
            \,\varphi^N=0 \text{ a.e. on } S_0 \text{ and } \varphi^N=1 \text{ a.e. on } S_1
            \right.
        \right\},
\end{align*}
to fix material on $S_0$ and complete void on $S_1$.

For $l\in\N$ and $i_1,\dots,i_l\in \mathbb{N}$, the eigenvalues $\lambda_{i_1},...,\lambda_{i_l}$ are to be penalized via a function
\begin{align*}
    \Psi: \left(\R_{>0}\right)^l\to \R, 
\end{align*}
which is assumed to be $C^1$ and bounded from below, i.e., we find a constant $c_{\Psi}>0$ such that $\Psi(\bx)\ge -c_{\Psi}$ for all $\bx\in \left(\R_{>0}\right)^l$. As mentioned above, we have to include the Ginzburg--Landau energy into our minimization problem. Hence we define the objective functional as
\begin{align}
\label{OBJ}
        J_{l}^{\eps}(\bphi)
    \coloneqq
        \Psi(\lambda_{i_1}^{\bphi},\dots, \lambda_{i_l}^{\bphi})
    +
        \gamma E^{\eps}(\bphi),
\end{align}
with $\gamma>0$. Consequently, the overall optimization problem reads as
\begin{align}\label{Pepsla} \tag{$\mathcal{P}^{\eps}_{l}$}
    \left\{
        \begin{aligned}
            &\text{ min} 
            &&J^{\eps}_l(\bphi),\\
            &\text{ s.t.} 
            &&\bphi\in \mathcal{\B{\mathcal{G}}}^{\B{m}}\cap \B{U}_c,\\
            &
            &&\lambda^{\bphi}_{i_1},\dots, \lambda^{\bphi}_{i_l} 
            \;\text{are eigenvalues of } \eqref{WWP}.
        \end{aligned}
    \right.
\end{align}

To investigate this optimal control problem, we first have to establish the existence of eigenvalues along with suitable associated eigenfunctions. This topic is addressed in the next section.

\subsection{A combination of compliance and eigenvalue optimization}\label{COPT}

In \cite{Blank}, the problem of minimizing the mean compliance
\begin{align*}
            F(\bu,\bphi)
        =
            \int_{\Omega}
                \big(1-\varphi^N\big)\B{f}\cdot\bu
            \text{\,d}x
        +
            \int_{\Gamma_g}\B{g}\cdot\bu\,\text{d}\Gamma, 
\end{align*}
with $\B{f}\in \Lzd$ and $\B{g}\in L^2(\Gamma_g,\mathbb{R}^d)$, and the deviation with respect to a target displacement $\bu_{\Omega}\in \Lzd$ given by 
\begin{align*}
        J_0(\bu,\bphi)
    =
        \left(
            \int_{\Omega}c\big(1-\varphi^N\big)\abs{\bu-\bu_{\Omega}}^2\text{\,d}x
        \right)^{\nu},
         \quad \nu\in (0,1],
\end{align*}
is also considered. Here, $c\in L^{\infty}(\Omega)$ denotes a function with $\abs{\text{supp\,}c}>0$, where $\abs{\text{supp\,}c}$ stands for the Lebesgue measure of the support. 
The boundary $\partial\Omega$ is split into two relatively open, disjoint subsets $\Gamma_C,\Gamma_g\subset \partial\Omega$ such that $\partial\Omega=\overline{\Gamma_C\cup\Gamma_g}$.
Moreover, the state equation is determined by the mean compliance in order to obtain 
\begin{align*}
	\bu\in \HC\coloneqq \Big\{ \bet\in H^1(\Omega;\R^d) \,\Big|\, \bet = \B 0 \;\text{a.e. on}\; \Gamma_C \Big\}
\end{align*}
as the displacement vector under the given forces. It reads as
\begin{align}\label{wState}
\begin{cases}
\begin{array}{rll}
-\nabla\cdot\left[\C(\bphi)\mathcal{E}(\bu)\right]
&=
\big(1-\varphi^N\big)\B{f}
& \quad\text{in }\Omega,\\
\bu
&=
\B{0}
&\quad\text{on }\Gamma_C,\\
\left[\C(\bphi)\mathcal{E}(\bu)\right]\B{n}
&=
\B{g}
&\quad\text{on }\Gamma_g.
\end{array}
\end{cases}
\end{align}

Combining this problem with the one discussed in Subsection~\ref{SOPT}, we obtain a structure that is on the one hand as stiff as possible (i.e., it has small compliance) and on the other hand realizes the desired vibration properties (e.g., a large first eigenvalue). 
In Section~7, we will present an existence result as well as the variational inequality for this combined problem.
Then the combination of $\left(\mathcal{P}^{\eps}\right)$ in \cite{Blank} and $\eqref{Pepsla}$ reads as
\begin{align}\tag{$\mathcal{K}^{\eps}_{l}$}\label{Kepsl}
    \left\{
        \begin{aligned}
                &\min &&I_l^{\eps}(\bu,\bphi)
            =
                \alpha F(\bu,\bphi)
                +\beta J_0(\bu,\bphi)
                +\gamma E^{\eps}(\bphi)
                +\Psi(\lambda^{\bphi}_{i_1},\dots,\lambda^{\bphi}_{i_l})\\
            &\text{ s.t.}
            &&(\bu,\bphi)
                \in H^1_{C}(\Omega,\R^d)\times H^1(\Omega,\R^N),\\[0.25ex]
            &&&
                \eqref{wState} \text{ is fulfilled}, 
                \bphi\in \mathcal{\B{\mathcal{G}}}^{\B{m}}\cap \B{U}_c,\\
            &&&\text{and } 
                \lambda^{\bphi}_{i_1},\dots,\lambda^{\bphi}_{i_l} 
                \text{ are}\text{ eigenvalues of } \eqref{WWP},
        \end{aligned}
    \right.
\end{align}
where $\alpha,\beta\ge 0$, $\gamma,\eps>0$, $\B{m}\in (0,1)^N\cap\Sigma^N$.


\section{Analysis of the state equation}

\begin{Def}[Definition of eigenvalues and eigenfunctions]\label{DEF:EEW}
	Let $\bphi\in \LuN$ be arbitrary. Then $\lambda^\bphi$ is called an \emph{eigenvalue} of the state equation \eqref{EWP} if there exists a nontrivial weak solution $\bw^\bphi$ to the system \eqref{EWP}, i.e., $\B 0\neq \bw^\bphi \in \Hd$ and it holds that
	\begin{align}\label{WEstate}
	       \BE{\bw^{\bphi}}{\bet}
	   =
	       \lambda^{\bphi}
	       \big(\bw^{\bphi},\bet\big)_{\rho(\bphi)}
	\quad \text{for all $\bet\in H^1_D(\Omega,\R^d)$}.
	\end{align}
	In this case, the function $\bw^\bphi$ is called an \emph{eigenfunction} to the eigenvalue $\lambda^\bphi$.
\end{Def}

The assumptions of the previous section allow us to prove two classical functional analytic results in our setting.
\begin{Thm}[Existence and properties of eigenvalues and eigenfunctions]\label{EEW}
	$\,$\newline Let ${\bphi\in \LuN}$ be arbitrary.
	\begin{enumerate}[label=$\mathrm{(\alph*)}$, ref = $\mathrm{\alph*}$, leftmargin=*]
		\item
		There exists a sequence
		\begin{align*}
		  \left(
		      \bw^{\bphi}_k,\lambda_k^{\bphi}
		  \right)_{k\in \N}
		\subset\Hd\times \R
		\end{align*}
		possessing the following properties:
		\begin{itemize}
			\item For all $k\in\N$, $\bw^{\bphi}_k$ is an eigenfunction to the eigenvalue $\lambda^\bphi_k$ in the sense of Definition~\ref{DEF:EEW}.
			\item The eigenvalues $\lambda_k^{\bphi}$ (which are repeated according to their multiplicity) can be ordered in the following way:
			\begin{align*}
    			0<\lambda_1^{\bphi}
    			\le\lambda_2^{\bphi}
    			\le\lambda_3^{\bphi}
    			\le\cdots.
			\end{align*}
			Moreover, it holds that $\lambda_k^{\bphi}\to \infty$ as $k\to\infty$, and there exist no further eigenvalues of the state equation $\eqref{WEstate}$.
			\item The eigenfunctions $\{\bw^{\bphi}_1, \bw^{\bphi}_2,\dots\} \subset \Hd$ form an $\Lzrp$-orthonormal basis of the space $\Lzrp$.
		\end{itemize}
		\item
		For $k\in\N$ we have the Courant--Fischer characterization
		\begin{align*}
		      \lambda_k^{\bphi}
		  =
               \underset{V\in \mathcal S_{k-1}}{\max}\min 
		          \left\{
    		          \left. 
    		              \frac{\BE{\bu}{\bu}}{\norm{\bu}^2_{\Lzrp}}
    		          \right\vert
		              \begin{aligned}
		                  &\bu\in V^{\perp,\Lzrp}\cap \Hd, \\
		                  &\bu\neq \B{0}
		              \end{aligned}
		          \right\}.
		\end{align*}
		Here, $\mathcal S_{k-1}$ denotes the collection of all $(k-1)$-dimensional subspaces of $\Hd$. The set $V^{\perp,\Lzrp}$ denotes the orthogonal complement of $V\subset \Lzd$ with respect to the scalar product on $\Lzrp$.
		
		Moreover, the maximum is attained at the subspace
		\begin{align*}
		  V=\langle\bw^{\bphi}_1,\dots,\bw^{\varphi}_{k-1}\rangle_{\textup{span}}.
		\end{align*}
	\end{enumerate}
\end{Thm}
\begin{proof}
	Using the Lax--Milgram theorem and the fact that $\mathcal{H}^{d-1}\left(\Gamma_D\right) > 0$, we conclude that for any $\B f\in \Lzd$, there exists a unique function ${\B v}_{\B f} \in \Hd$ solving the equation
	\begin{align*}
	       \BE{\B{v}_{\B f}}{\bet}
	   =
	       \int_{\Omega}\rho(\bphi)\B{v}_{\B{f}}\cdot\bet\text{\,d}x 
	       \quad \text{for all $\bet\in H^1_D(\Omega,\R^d)$}.
	\end{align*}
	This allows us to define a solution operator 
	\begin{align*}
	   \mathcal T: \Lzd \to \Hd \subset \Lzd,\quad \B f\mapsto {\B v}_{\B f}.
	\end{align*}
	Since $\Hd$ is compactly embedded in $\Lzd$, we can easily show that $\mathcal T$ is a compact, self-adjoint, and bounded linear operator.
	Thus, the assertions in (a) directly follow from the spectral theorem for compact self-adjoint operators (see, e.g., \cite[Sect.~12.12]{Alt}).
	
	To prove (b), we first infer from \eqref{WEstate} that the sequence 
	\begin{align*}
	   \left(\frac{\bw_k^\bphi}{\sqrt{\lambda_k}}\right)_{k\in \N} \subset \Hd,
	\end{align*}
	forms an orthonormal basis of $\Hd$ when taking the inner product \eqref{eq:newinner}. For any $\B{v}\in \Hd$, this yields the representation 
	\begin{align*}
	   \B{v}=\sum_{i=1}^{\infty}\rsp{\B{v}}{\bw_i^{\bphi}}{\bphi}\bw_i^{\bphi},
	\end{align*}
	where the series on the right-hand side converges in $\Hd$. In the following, we will sometimes omit the exponent $\bphi$ for a more convenient depiction.
	
	To establish the Courant--Fischer characterization we now fix an arbitrary subspace $V\in \mathcal S_{k-1}$. Let us denote the orthogonal projection from $\Lzrp$ to $V$ with respect to the scalar product on $\Lzrp$ by
	\begin{align*}
	   P_{\bphi}: \Lzrp \to V\subset \Lzrp.
	\end{align*}
	Since $V$ is a $(k-1)$-dimensional subspace, the family
	\begin{align*}
	   \left\{ 
	       P_{\bphi}(\bw_1),\dots,P_{\bphi}(\bw_k)
	   \right\}\subset V,
	\end{align*}
	must be linearly dependent. Hence, for every $i\in \left\{1,\dots,k\right\}$, we find coefficients $\alpha_i\in \R$ that are not all equal to zero such that
	\begin{align*}
    	   P_{\bphi}\left(\sum_{i=1}^{k}\alpha_i\bw_i\right)
    	=\sum_{i=1}^{k}\alpha_iP_{\bphi}(\bw_i)
    	=\B{0}.
	\end{align*}
	Per construction of the orthogonal projection this is equivalent to
	\begin{align*}
	   \B{v}\coloneqq \sum_{i=1}^{k}\alpha_i\bw_i\in V^{\perp, \Lzrp}\cap\Hd.
	\end{align*}
	As not all of the coefficients vanish, and since the eigenfunctions $\left\{\bw_1,\dots, \bw_k\right\}$ are linearly independent, we know that $\B{v}\neq \B{0}$.
	Using the orthogonality of eigenfunctions and the fact, that the sequence of eigenvalues increases, we conclude that
	\begin{align}
	   \begin{aligned}
	       \inf 
	           &\left\{
    	           \left. 
    	               \frac{\BE{\bu}{\bu}}{\norm{\bu}^2_{\Lzrp}}
	               \right\vert
	               \begin{aligned}
	                   &\bu\in V^{\perp,\Lzrp}\cap \Hd, \\
	                   &\bu\neq \B{0}
	               \end{aligned}
	              \right\}\label{infu}\\[2ex]
	        &\le 
	               \frac{\BE{\B{v}}{\B{v}}}{\norm{\B{v}}^2_{\Lzrp}}
	             =
	                \frac{\sum_{i=1}^{k}\alpha_i^2\lambda_i}{\sum_{i=1}^{k}\alpha_i^2}
	             \le
	                 \lambda_{k}.
	   \end{aligned}
	\end{align}
    As the infimum in \eqref{infu} obviously exists, we can find a minimizing sequence 
	\begin{align*}
	   \left(\bu_l\right)_{l\in \N}\subset \left\{V^{\perp,\Lzrp}\cap \Hd\right\}\backslash\{\B{0}\},
	\end{align*} 
	such that
	\begin{align*}
	       \underset{l\to \infty}{\lim}\frac{\BE{\bu_l}{\bu_l}}{\norm{\bu_l}^2_{\Lzrp}}
	   =
            \inf \left\{
	           \left. 
	               \frac{\BE{\bu}{\bu}}{\norm{\bu}^2_{\Lzrp}}
	           \right\vert
	            \begin{aligned}
	                   &\bu\in V^{\perp,\Lzrp}\cap \Hd, \\
	                   &\bu\neq \B{0}
	             \end{aligned}
	           \right\}.
	\end{align*}
	Now, recalling that the inner product $\BE{\cdot}{\cdot}$ induces a norm on $\Hd$, this implies that the sequence 
	\begin{align*}
    	   \left(\tilde{\bu}_l\right)_{l\in \N}
    	\coloneqq 
            \left(\frac{\bu_l}{\norm{\bu_l}_{\Lzrp}}\right)_{l\in \N},
	\end{align*} 
	is bounded in $\Hd$. Hence, due to the Banach--Alaoglu theorem and the compact embedding $\Hd \subset \Lzd$, there exists a function $\tilde{\B u}\in \Hd$ such that 
	\begin{alignat*}{3}
    	   \tilde{\bu}_l \rightharpoonup \tilde{\bu}\quad\text{in }\Hd,
    	\quad\text{and}\quad
    	   \tilde{\bu}_l \to \tilde{\bu}\quad\text{in } \Lzrp,
	\end{alignat*}
	along a non-relabeled subsequence.
	In particular, since all members of the sequence $\left(\tilde{\bu}_l\right)_{l\in \N}$ are normalized with respect to the $\Lzrp$-norm, it follows that $\tilde{\bu}\neq \B{0}$. 
	
	Furthermore, $V^{\perp,\Lzrp}\cap \Hd$ is a convex and closed subset of $\Hd$. Hence, it is also weakly (sequentially) closed and we thus know that 
	\begin{align*}
	   \tilde{\bu}\in \left\{V^{\perp,\Lzrp}\cap \Hd\right\}\backslash\{\B{0}\}.
	\end{align*}
	Using the fact that norms are always weakly lower semi-continuous we infer that $\tilde{\bu}$ is a minimizer of the expression in $\eqref{infu}$ via the direct method in the calculus of variations.
	
	Since this holds for any arbitrary $(k-1)$-dimensional subspace $V\subset \Hd$, we conclude that 
	\begin{align*}
	   \underset{V\in \mathcal{S}_{k-1}}{\sup}\min 
    	   \left\{
    	       \left. 
    	           \frac{\BE{\bu}{\bu}}{\norm{\bu}^2_{\Lzrp}}\right\vert
    	            \begin{aligned}
    	               &\bu\in V^{\perp,\Lzrp}\cap \Hd, \\
    	               &\bu\neq \B{0}
    	            \end{aligned}
    	   \right\}
	   \le \lambda_k.
	\end{align*} 
	We now select a special $(k-1)$-dimensional subspace defined by
	\begin{align*}
	   V\coloneqq \langle \bw_1,\dots \bw_{k-1}\rangle_{\text{span}}\subset \Hd.
	\end{align*}
	Then the definition of the orthogonal complement yields
	\begin{align*}
	   V^{\perp,\Lzrp}=\langle \bw_k,\bw_{k+1},\dots\rangle_{\text{span}}\subset \Lzrp.
	\end{align*}
	Hence, any $\B{v}\in V^{\perp,\Lzrp}\cap \Hd$ can be represented as
	\begin{align*}
	   \B{v}=\sum_{i=k}^{\infty}\rsp{\B{v}}{\bw_i}{\bphi}\bw_i,
	\end{align*}
	where the series converges in $\Hd$. Consequently, we obtain
	\begin{align*}
	   \BE{\B{v}}{\B{v}}
	=
        \sum_{i=k}^{\infty}
	       \rsp{\B{v}}{\bw_i}{\varphi}^2\lambda_i
	     \ge
	       \lambda_k\norm{\B{v}}_{\Lzrp}^2,
	\end{align*}
	because of the identity 
	\begin{align*}
    	   \sum_{i=k}^{\infty}
    	       \rsp{\B{v}}{\bw_i}{\varphi}^2
    	=
    	   \norm{\B{v}}_{\Lzrp}^2.
	\end{align*}
	Altogether, we conclude that
	\begin{align*}
    	\underset{\substack{V\subset \mathcal{S}_{k-1}\\\dim(V)=k-1}}{\sup}\min 
    	   \left\{
    	       \left.\frac{\BE{\bu}{\bu}}{\norm{\bu}^2_{\Lzrp}}
    	       \right|
    	       \begin{aligned}
    	           &\bu\in V^{\perp,\Lzrp}\cap \Hd, \\
    	           &\bu\neq \B{0}
    	       \end{aligned}
    	   \right\}
    	=\lambda_k.
	\end{align*}
	This means that the maximum is attained at the subspace $V=\langle \bw_1,\dots \bw_{k-1}\rangle_{\text{span}}$ at $\B{0}\neq\bw_k\in V^{\perp,\Lzrp}\cap \Hd$, which proves the claim.
\end{proof}


\section{Continuity of the eigenvalues and the associated eigenfunctions}

\subsection{Weak sequential continuity of the eigenvalues} \label{WSC}

First of all we only consider the first eigenvalue $\lambda_1$ to establish continuity results with respect to the phase-field $\bphi$. Afterwards, we proceed inductively to obtain these results also for all the other eigenvalues.

We consider the mapping
\begin{align*}
    \lambda_1: \HN\cap\LuN \to \R_{>0},\quad
    \bphi\mapsto \ev
\end{align*}
associated with the first eigenvalue.

The first continuity result for the eigenvalue $\lambda_1$ is obtained by proving lower and upper semi-continuity. Lower semi-continuity is established by the following lemma.

\begin{Lem}\label{lus}
	Let $\left(\bphi_k\right)_{k\in\N}\subset \HN\cap \LuN$ be a bounded sequence with respect to the $\LuN$-norm satisfying 
	\begin{align*}
	   \bphi_k\rightharpoonup \bphi\quad \text{in $\HN$ as $k\to\infty$}.
	\end{align*}
	Then it holds that 
	\begin{align*}
	   \ev\le \underset{k\to\infty}{\lim\inf}\lambda^{\bphi_k}_1,
	\end{align*}
	along a non-relabeled subsequence.
\end{Lem}
\begin{proof}
	We first notice that the assumptions of Lemma~\ref{lus} imply that $\bphi\in \LuN$. Let $\left\{\bw_1,\bw_2,\dots\right\}\subset \Lzrp$ denote an orthonormal basis of eigenfunctions corresponding to the sequence of eigenvalues $\left(\lambda_i^{\bphi}\right)_{i\in \N}$ from Theorem~\ref{EEW}.
	
	Now, for any $k\in \N$, we choose an arbitrary $\Lzrpa{\bphi_k}$-normalized eigenfunction $\bu^k$ that fulfills $\eqref{WEstate}$ for $\lambda_1^{\bphi_k}$. This choice is not necessarily unique up to multiplication with $\pm 1$, as we do not assume simplicity of $\lambda_1^{\bphi_k}$ or $\lambda_1^{\bphi}$ yet.
	
	Using the Courant--Fischer representation from Theorem~\ref{EEW}(b) for the first eigenvalue and the continuity of $\C$ and $\rho$, we see that the sequence $(\bu^k)_{k\in \N}\subset \Hd$ is bounded. By the Banach--Alaoglu theorem and the compact embedding $\Hd\subset \Lzd$, we infer the existence of a function $\oB{u}\in \Hd$ with
	\begin{alignat}{3}\label{wsco}
    	\begin{aligned}
    	       \bu^k\rightharpoonup \oB{u}\quad \text{in }\Hd,
    	   \quad\text{and}\quad
    	       \bu^k\to \oB{u}\quad \text{in }\Lzrp,
    	\end{aligned}	
	\end{alignat}
	as $k\to \infty$, up to a subsequence. With the help of Lebesgue's theorem and the assumptions on the sequence $\bphi_k$, this yields 
	\begin{align*} 
    	   \brsq{\bu^k}{\bu^k}{\bphi_k}
    	\to 
    	   \brsq{\oB{u}}{\oB{u}}{\bphi},
	\end{align*}
	as $k\to \infty$ after another subsequence extraction. This implies $\norm{\oB{u}}_{\Lzrp}=1$ since the members $\bu^k$ were chosen as $\Lzrpa{\bphi_k}$-normalized eigenfunctions. 
	In particular, this implies that 
	\begin{align}
	   1=\sum_{i=1}^{\infty}\rsp{\overline{\bu}}{\bw_i}{\varphi}^2.
	\end{align}
	Plugging $\oB u$ into the continuous bilinear form $\BE{\cdot}{\cdot}$ on $\Hd$, and invoking the increasing order of the sequence $\left(\lambda_i^{\bphi}\right)_{i\in \N}$, we conclude that
	\begin{align*}
    	   \BE{\oB{u}}{\oB{u}}
    	=
        	\sum_{i=1}^{\infty}
        	 \rsq{\overline{\bu}}{\bw_i}{\varphi}^2\lambda^{\bphi}_i
    	\ge
    	      \lambda^{\bphi}_1.
	\end{align*}
	
	If we can now show that
	\begin{align}\label{suse}
	   \underset{k\to\infty}{\lim\inf}\lambda^{\bphi_k}_1\ge \BE{\oB{u}}{\oB{u}},
	\end{align}
	the proof would be complete.
	Using the convergence results we have just established, the Cauchy--Schwarz inequality and the weak formulation \eqref{WEstate}, we infer that
	\begin{alignat*}{2}
    	   \underset{k\to\infty}{\lim\inf}\lambda^{\bphi_k}_1
    	&=
        	\left(\underset{k\to\infty}{\lim\inf}\lambda^{\bphi_k}_1\right)
        	\left(\underset{k\to\infty}{\lim}\norm{\bu^k}_{\Lzrpa{\bphi_k}}\right)
        	\left(\underset{k\to\infty}{\lim}\norm{\oB{u}}_{\Lzrpa{\bphi_k}}\right)\\[1ex]
    	&=
        	\underset{k\to\infty}{\lim\inf}\left[\lambda^{\bphi_k}_1
        	\norm{\bu_k}_{\Lzrpa{\bphi_k}}
        	\norm{\oB{u}}_{\Lzrpa{\bphi_k}}\right]\\[1ex]
    	&\ge 			
        	\underset{k\to\infty}{\lim\inf}
        	\left[
        	   \lambda^{\bphi_k}_1\brsp{\bu^k}{\oB{u}}{\bphi_k}
        	\right]\\[1ex]
    	&=
    	   \underset{k\to\infty}{\lim\inf}\;\bBET{\bu^k}{\oB{u}}{\bphi_k}.
	\end{alignat*}
	We further know that
	\begin{align*}
    	\begin{aligned}
    	       &\bBK{\bu^k}{\oB{u}}{k}-\bBE{\oB{u}}{\oB{u}}\\
    	   &\quad=
    	       \left[
    	           \bBK{\bu^k}{\oB{u}}{k}-\bBE{\bu^k}{\oB{u}} 
    	       \right] \\
    	   &\qquad +
    	       \left[
    	           \bBE{\bu^k}{\oB{u}}-\bBE{\oB{u}}{\oB{u}}
    	       \right].
    	\end{aligned}
	\end{align*}
	Using Lebesgue's convergence theorem, the boundedness of $(\bu^k)_{k\in \N}$, the local Lipschitz continuity of $\C$ and the assumptions on $\left(\bphi_k\right)_{k\in \N}$, we conclude that the first summand converges to zero along a non-relabeled subsequence. The second summand converges to zero as a direct consequence of $\eqref{wsco}$.
	
	In summary, we obtain that
	\begin{align*}
    	   \underset{k\to\infty}{\lim\inf}\lambda^{\bphi_k}_1
    	\ge
    	   \underset{k\to\infty}{\lim\inf}\bBET{\bu^k}{\oB{u}}{\bphi_k}
    	=
    	   \bBE{\oB{u}}{\oB{u}}\ge \lambda^{\bphi}_1.
	\end{align*}
	which completes the proof.
\end{proof}

Now, we establish the corresponding result for weak upper semi-continuity.

\begin{Lem}\label{los}
	Let $\left(\bphi_k\right)_{k\in\N}\subset \HN\cap \LuN$ be a bounded sequence with respect to the $\LuN$-norm satisfying 
	\begin{align*}
	   \bphi_k\rightharpoonup \bphi\quad \text{in $\HN$ as $k\to\infty$}.
	\end{align*}
	Then it holds that 
	\begin{align*}
	   \ev\ge \underset{k\to\infty}{\lim\sup\,}\lambda^{\bphi_k}_1,
	\end{align*}
    along a non-relabeled subsequence.
\end{Lem}
\begin{proof}
	As $\C$ and $\rho$ satisfy suitable continuity properties we can proceed as in \cite[Thm.~8.1.3]{Henrot} and use once more the Courant--Fischer representation for the first eigenvalue to prove the claim.
\end{proof}

Combining both lemmata we can conclude that $\lambda_1$ is weakly sequentially continuous.

\begin{Cor}\label{Wst}
	Let $\left(\bphi_k\right)_{k\in\N}\subset \HN\cap \LuN$ be a bounded sequence with respect to the $\LuN$-norm satisfying 
	\begin{align*}
	   \bphi_k\rightharpoonup \bphi\quad \text{in $\HN$ as $k\to\infty$},
	\end{align*} 
	and let $(\bu^k)_{k\in\N}\subset \Hd$ be a sequence of $\Lzrpa{\bphi_k}$-normalized eigenfunctions to the eigenvalues $(\lambda_1^{\bphi_k})_{k\in\N}$, i.e., $\bu_k$ satisfies $\eqref{WEstate}$ written for $\lambda_{1}^{\bphi_k}$ for every $k\in\N$.
	
	Then it holds that 
	\begin{align}\label{lko}
	   \lambda^{\bphi_k}_1\to \ev, \quad\text{as $k\to\infty$},
	\end{align}
	i.e., the \emph{whole} sequence of eigenvalues converges and not just a subsequence.
	
	Furthermore, there exists a $\Lzrp$-normalized eigenfunction $\oB{u}\in \Hd$ to the eigenvalue $\lambda_1^{\bphi}$  such that
	\begin{alignat*}{2}
    	   \bu^k\rightharpoonup\oB{u}\quad \text{in }\Hd,
    	\quad\text{and}\quad
    	   \bu^k\to \oB{u}\quad \text{in }\Lzrp,
	\end{alignat*}
	as $k\to \infty$, along a non-relabeled subsequence.
\end{Cor}

\begin{proof}
	The convergence $\lambda_1^{\bphi_k} \to \lambda_1^{\bphi}$ as $k\to\infty$ follows from Lemma~\ref{lus} and Lemma~\ref{los} after extraction of a subsequence. Moreover, as the limit $\lambda_1^{\bphi}$ does not depend on the choice of the subsequence, we conclude by a standard contradiction argument that the convergence remains true for the whole sequence.
	
	The convergence properties of $(\bu^k)_{k\in\N}\subset \Hd$ and the fact that the weak limit $\oB{u}\in \Hd$ is $\Lzrp$-normalized have already been established in $\eqref{wsco}$. Hence, it remains to show that $\oB{u}\in \Hd$ is an eigenfunction corresponding to the eigenvalue $\lambda_{1}^{\bphi}$.
	By construction, we know that for any $k\in \N$,
	\begin{align}\label{Ek}
    	   \bBET{\bu^k}{\bet}{\bphi_k}
    	&=
    	   \lambda^{\bphi_k}_1
    	   \brsp{\bu^k}{\bet}{\bphi_k},
	\end{align}
	for all test functions $\bet\in \Hd$. Using the convergence of eigenvalues $\eqref{lko}$ and proceeding as in the proof of Lemma~\ref{lus}, we infer that for any any $\bet\in \Hd$, %
	\begin{align*}
    	\bBET{\bu^k}{\bet}{\bphi_k}&\to \BE{\oB{u}}{\bet},\\
    	\brsq{\bu^k}{\bet}{\bphi_k}&\to \brsq{\oB{u}}{\bet}{\bphi},
	\end{align*}
	as $k\to \infty$, after extraction of a subsequence. Hence, we can pass to the limit in equation $\eqref{Ek}$ to obtain 
	\begin{align*}
	       \BE{\oB{u}}{\bet}
	   =
        	\lambda_1^{\bphi}
        	\brsp{\oB{u}}{\bet}{\bphi},
	\end{align*}
	which proves that $\oB{u}\in \Hd$ is indeed an eigenfunction corresponding to $\lambda_1^{\bphi}$.
\end{proof}


Corollary~\ref{Wst} now serves as initial case for the following inductive proof which yields convergence of all eigenvalues.

\begin{Thm}[Continuity properties for the eigenvalues and their eigenfunctions]\label{slw}
	Let $j\in \N$ be arbitrary and
	let $\left(\bphi_k\right)_{k\in\N}\subset \HN\cap \LuN$ be a bounded sequence with respect to the $\LuN$-norm satisfying 
	\begin{align*}
	   \bphi_k\rightharpoonup \bphi\quad \text{in $\HN$ as $k\to\infty$}.
	\end{align*} 
	Moreover, let $(\bu_j^k)_{k\in\N}\subset \Hd$ be a sequence of $\Lzrpa{\bphi_k}$-normalized eigenfunctions to the eigenvalues $(\lambda_{j}^{\bphi_k})_{k\in\N}$,
	i.e., $\bu_j^k$ satisfies $\eqref{WEstate}$ written for $\lambda_{j}^{\bphi_k}$ for every $k\in\N$.
	
	Then it holds that
	\begin{align*}
	   \lambda^{\bphi_k}_j\to \lambda^{\bphi}_j,\quad \text{as $k\to\infty$},
	\end{align*}
	for the \emph{whole} sequence of eigenvalues and not just a subsequence.
	
	\pagebreak[1]
	
	Furthermore, there exists a $\Lzrp$-normalized eigenfunction $\oB{u}_{j}\in \Hd$ to the eigenvalue $\lambda_j^\bphi$ such that
	\begin{align*}
    	   \bu^{k}_{j}\rightharpoonup\oB{u}_j\quad \text{in }\Hd,
    	\quad\text{and}\quad
    	   \bu^{k}_j\to \oB{u}_j\quad \text{in }\Lzrp
	\end{align*}
	as $k\to \infty$ along a non-relabeled subsequence.
\end{Thm}

\begin{proof}
	As mentioned before we proceed by induction. The initial step has already been established in Corollary~\ref{Wst}.
	
	Now, we assume that the statement is already verified for the index $\left(j-1\right)\in \mathbb{N}$. 
	Our task is to prove that the assertion is true for the $j$-th eigenvectors and the associated eigenfunctions. In this regard, the Courant--Fischer representation of Theorem~\ref{EEW}(b) will be a helpful tool.
	
	For $k\in \mathbb{N}$, we fix the $(j-1)$-dimensional subspace of $\Hd$ that realizes the maximum in the Courant--Fischer representation discussed in Theorem~\ref{EEW}(b), namely
	\begin{align*}
    	   V_k
    	\coloneqq 
    	   \langle 
    	       \B{w}^{\B{\varphi}_k}_1,\dots, \B{w}^{\B{\varphi}_k}_{j-1}
    	   \rangle_{\text{span}}.
	\end{align*}
	Analogously, we define
	\begin{align*}
    	   V
    	\coloneqq 
    	   \langle 
    	       \B{w}^{\B{\varphi}}_1,\dots, \B{w}^{\B{\varphi}}_{j-1}
    	   \rangle_{\text{span}}.
	\end{align*}
    Then by the induction hypothesis we know that for every $i=1,\dots,j-1$, there exists a $\Lzrp$-normalized eigenfunction $\oB{u}_i\in\Hd$ to the eigenvalue $\lambda_i^{\bphi}$ such that
    \begin{align}\label{indhyp}
    	   \B{w}^{\bphi_k}_{i}\rightharpoonup\oB{u}_i\quad \text{in }\Hd,
    	\quad\text{and}\quad
    	   \B{w}^{\bphi_k}_i\to \oB{u}_i\quad \text{in }\Lzrp,
    \end{align}
    as $k\to \infty$ along a non-relabeled subsequence. As $\left\{\B{w}^{\B{\varphi}_k}_1,\B{w}^{\B{\varphi}_k}_2,\dots \right\}\subset\Lzrpa{\bphi_k}$ form an orthonormal basis we infer that
    \begin{align}\label{uorth}
        \rsp{\oB{u}_m}{\oB{u}_l}{\bphi}=0,
    \end{align}
    for $m\neq l$, using the convergence properties of the sequence $\left(\bphi_k\right)_{k\in \mathbb{N}}$ along with Lebesgue's convergence theorem. In particular, the family $\left\{\oB{u}_1,\dots,\oB{u}_{j-1}\right\}\subset \Lzrp$ is linearly independent, which yields that all eigenfunctions to eigenvalues strictly smaller than $\lambda_j^{\B{\varphi}}$ are contained in its span $W\coloneqq \langle \oB{u}_1,\dots,\oB{u}_{j-1}\rangle_{\text{span}}$. Hence, we conclude that
    \begin{align}\label{CFj}
        \min\left\{
        	   \left.
        	       \frac{\BET{\B{u}}{\B{u}}{\B{\varphi}}}{\norm{\B{u}}_{\Lzrpa{\B{\varphi}}}^2}
        	   \right|
        	       \B{u}\in W^{\perp,\Lzrpa{\B{\varphi}}}\cap \Hd, \B{u}\neq \B{0}
        	   \right\}\ge \lambda_j^{\bphi}.
    \end{align}
    As the minimum is attained, we infer that we find a non-trivial function $\B{v}\in \Hd$ with
    \begin{align}\label{condv}
        \rsp{\B{v}}{\oB{u}_i}{\bphi}=0,
    \end{align}
    for all $i=1,\dots,j-1$ such that
    \begin{align}\label{cond2v}
        \frac{\BET{\B{v}}{\B{v}}{\B{\varphi}}}{\norm{\B{v}}_{\Lzrpa{\B{\varphi}}}^2}=\lambda_j^{\bphi}.
    \end{align}
    Otherwise the inequality in $\eqref{CFj}$ would be strict, which would be a contradiction to Theorem~\ref{EEW}.
	This means we have shown the existence of a function
	\begin{align}
	\label{ASS:V}
		  \B{v}\in W^{\perp,\Lzrp}\cap \Hd 
		\quad\text{with}\quad 
		  \B{v}\neq \B{0},
	\end{align}
    fulfilling $\eqref{cond2v}$.
    
    Let now the sequence $(\B v_k)_{k\in\N}$ be defined by
	\begin{align}\label{ffol}
    	   \B{v}_k
    	\coloneqq
    	   \B{v}
    	-\sum_{i=1}^{j-1}
    	   \rsq{\B{v}}{\B{w}^{\B{\varphi}_k}_i}{\B{\varphi}_k}\B{w}^{\B{\varphi}_k}_i.
	\end{align}
	for all $k\in\N$. By this construction, we immediately observe that 
	\begin{align*}
	   \B{v}_k\in V_k^{\perp,\Lzrpa{\B{\varphi}_k}}\cap \Hd.
	\end{align*}
	We now intend to show that the convergences
	\begin{align}
	\label{epk}
	       \BET{\B{v}_k}{\B{v}_k}{\B{\varphi}_k}
	   &\to \BET{\B{v}}{\B{v}}{\B{\varphi}}, \\
	       \label{vpk}
	       \norm{\B{v}_k}_{\Lzrpa{\B{\varphi}_k}}
	   &\to \norm{\B{v}}_{\Lzrp},
	\end{align}
	as $k\to \infty$, hold along a non-relabeled subsequence. 
	
	To verify \eqref{epk}, we consider the decomposition
	\begin{align}
    	\label{ID:1}
    	\begin{aligned}
    	       &\BET{\B{v}_k}{\B{v}_k}{\B{\varphi}_k}\\
    	   &\quad=
    	       \BET{\B{v}}{\B{v}}{\B{\varphi}_k}
    	    -
    	       2\BET{\B{v}}{\sum_{i=1}^{j-1}
    		  \rsq{\B{v}}{\B{w}^{\B{\varphi}_k}_i}{\B{\varphi}_k}
    		  \B{w}^{\B{\varphi}_k}_i}{\B{\varphi}_k}\\
    	   &\qquad+
    	       \BET{
    		      \sum_{i=1}^{j-1}
    		      \rsq{\B{v}}{\B{w}^{\B{\varphi}_k}_i}{\B{\varphi}_k}
    		      \B{w}^{\B{\varphi}_k}_i
                   	}
    	           {\sum_{m=1}^{j-1}
    		          \rsp{\B{v}}{\B{w}^{\B{\varphi}_k}_m}{\B{\varphi}_k}
    		          \B{w}^{\B{\varphi}_k}_m
    	           }
    	           {\B{\varphi}_k}.
    	\end{aligned}
	\end{align}
	For the first product on the right-hand side, we directly obtain the convergence
	\begin{align*}
	   \BET{\B{v}}{\B{v}}{\B{\varphi}_k}\to \BET{\B{v}}{\B{v}}{\B{\varphi}}, \quad\text{as $k\to\infty$},
	\end{align*}
	along a suitable subsequence.
	As the functions $\B w_i^{\bphi_k}$ are $\Lzrpa{\B{\varphi}_k}$-normalized eigenfunctions, we obtain from \eqref{WEstate} the following representation of the third product on the right-hand side of \eqref{ID:1}:
	\begin{align*}
    	&\BET{
    		      \sum_{i=1}^{j-1}
    		      \rsq{\B{v}}{\B{w}^{\B{\varphi}_k}_i}{\B{\varphi}_k}\B{w}^{\B{\varphi}_k}_i
    	       }
    	       {\sum_{m=1}^{j-1}
    		      \rsq{\B{v}}{\B{w}^{\B{\varphi}_k}_m}{\B{\varphi}_k}\B{w}^{\B{\varphi}_k}_m
    	       }      
    	       {\B{\varphi}_k}
    	\\
    	&\quad =
    	       \sum_{i=1}^{j-1}
    	       \rsq{\B{v}}{\B{w}^{\B{\varphi}_k}_i}{\B{\varphi}_k}^2\lambda^{\B{\varphi}_k}_i. 
	\end{align*}
	As the sum takes only the indices $i=1,\dots, j-1$ into account, we can again use the induction hypothesis to obtain 
	\begin{align*}
        	   \sum_{i=1}^{j-1}
        	       \rsq{\B{v}}{\B{w}^{\B{\varphi}_k}_i}{\B{\varphi}_k}^2\lambda^{\B{\varphi}_k}_i
    	   \to
            	\sum_{i=1}^{j-1}
            	   \rsq{\B{v}}{\oB{u}_i}{\B{\varphi}}^2\lambda^{\B{\varphi}}_i,
            	\quad\text{as $k\to\infty$},
	\end{align*}
	along a suitable subsequence.
	Hence, $\eqref{condv}$ directly yields that the third product on the right-hand side of \eqref{ID:1} converges to zero.
	The second product can be handled similarly, and we can also show that it tends to zero as $k\to\infty$. In summary, we get
	\begin{align*}
	   \BET{\B{v}_k}{\B{v}_k}{\B{\varphi}_k}\to \BET{\B{v}}{\B{v}}{\B{\varphi}}, \quad\text{as $k\to\infty$}.
	\end{align*}
	This proves \eqref{epk}.
	The claim \eqref{vpk} can easily be verified using the induction hypothesis. 
	
	In particular, since $\B{v}\neq \B{0}$, we obtain that $\B{v}_k\neq \B{0}$ for all $k\in \mathbb{N}$ sufficiently large. For such $k\in\N$, we obtain the estimate
	\begin{align*}
    	\lambda^{\B{\varphi}_k}_j
    	   &=\min
    	       \left\{
    	           \left.
    	               \frac{\BET{\B{u}}{\B{u}}{\B{\varphi}_k}}{\norm{\B{u}}_{\Lzrpa{\B{\varphi}_k}}^2}
    	           \right|
    	               \B{u}\in V_k^{\perp,\Lzrpa{\B{\varphi}_k}}\cap \Hd, \B{u}\neq \B{0}
    	       \right\}\\[2ex]
    	   &\le 
    	       \frac{\BET{\B{v}_k}{\B{v}_k}{\B{\varphi}_k}}{\norm{\B{v}_k}_{\Lzrpa{\B{\varphi}_k}}^2}\;.
	\end{align*}
	Using \eqref{epk} and \eqref{vpk}, we conclude from $\eqref{cond2v}$ that, along a non-relabeled subsequence,
	\begin{align}\label{lisu}
    	   \underset{k\to \infty}{\limsup\,}
    	       \lambda^{\B{\varphi}_k}_j
    	\le 
    	   \underset{k\to \infty}{\limsup\,}
    	       \frac{\BET{\B{v}_k}{\B{v}_k}{\B{\varphi}_k}}{\norm{\B{v}_k}_{\Lzrpa{\B{\varphi}_k}}^2}
    	= 
    	   \frac{\BET{\B{v}}{\B{v}}{\B{\varphi}}}{\norm{\B{v}}_{\Lzrpa{\B{\varphi}}}^2}
    	=
    	   \lambda^{\B{\varphi}}_j.
	\end{align}
	In particular, this implies that the subsequence $(\lambda^{\B{\varphi}_k}_j)_{k\in \mathbb{N}}$ is bounded. 
    
	Now, for $k\in\N$, let $\B{u}^{k}_j$ denote a $\Lzrpa{\B{\varphi}_k}$-normalized eigenfunction to the eigenvector $\lambda^{\B{\varphi}_k}_j$.
	Consequently, due to the Courant--Fischer characterization in Theorem~\ref{EEW}(b), the sequence $(\B{u}_j^k)_{k\in \mathbb{N}}\subset \Hd$ is bounded. Applying the Banach--Alaoglu theorem, we can thus extract a subsequence such that
	\begin{align}\label{wlk}
    	   \B{u}^{k}_j\rightharpoonup \oB{w} \quad\text{in }\Hd,
    	\quad\text{and}\quad
    	   \B{u}^{k}_j\to \oB{w} \quad\text{in }\Lzd,
	\end{align}
	as $k\to\infty$, where $\oB{w}\in \Hd$ is $\Lzrp$-normalized. However, it is a priori not necessarily an eigenfunction to the eigenvalue $\lambda_{j}^{\B{\varphi}_k}$, as the convergence of the corresponding eigenvalues is still unknown.
	
	Proceeding as in Subsection~\ref{WSC}, we want to show that
	\begin{align}
    \label{liin}
    	   \lambda^{\B{\varphi}}_j
    	\le 
    	   \underset{k\to \infty}{\lim\inf }
    	       \lambda^{\B{\varphi}_k}_j.
	\end{align}
	As in the proof of Corollary~\ref{Wst} in combination with \eqref{lisu}, we can then conclude the desired convergence $\lambda_j^{\B{\varphi}_k}\to \lambda_j^{\B{\varphi}_k}$ for the whole sequence as $k\to\infty$.
	
	To verify \eqref{liin}, we first observe that due to the orthogonality of eigenfunctions corresponding to different eigenvalues
	\begin{align*}
	   \brsq{\B{u}^{k}_j}{\B{w}^{\B{\varphi}_k}_m}{\B{\varphi}_k}=0,
	\end{align*}
	for all $k\in \mathbb{N}$ and for all $m=1,\dots, j^*$, where $j^*<j-1$ is the maximal index such that $\lambda_{j^*}^{\bphi}<\lambda_{j-1}$.
	Recalling the assumptions on $\left(\B{\varphi}_k\right)_{k\in \mathbb{N}}$, we can use $\eqref{wlk}$ and $\eqref{indhyp}$ to infer that
	\begin{align*}
		\brsq{\oB{w}}{\oB{u}_m}{\B{\varphi}}=0,
	\end{align*}
	for all $m=1,\dots, j^*$. As $j^*$ is chosen maximally, we know from the orthogonality \eqref{uorth} that
    \begin{align*}
            \langle \oB{u}_1,\dots,\oB{u}_{j^*}\rangle_{\text{span}}
        =
            \langle \B{w}_1^{\bphi},\dots,\B{w}_{j^*}^{\bphi}\rangle_{\text{span}}\subset \Lzrp.
    \end{align*}
	This leads to the representation
	\begin{align*}
    	   \oB{w}
    	=
    	   \sum_{m=1}^{\infty}
    	       \rsp{\oB{w}}{\B{w}^{\B{\varphi}}_m}{\B{\varphi}}\B{w}^{\B{\varphi}}_m
    	=
    	   \sum_{m=j^*}^{\infty}
    	       \rsp{\oB{w}}{\B{w}^{\B{\varphi}}_m}{\B{\varphi}}\B{w}^{\B{\varphi}}_m.
	\end{align*}
	As the series converges in $\Hd$, we can use \eqref{WEstate} to obtain
	\begin{align*}
    	   \BE{\oB{w}}{\oB{w}}
    	=
    	   \sum_{m=j^*}^{\infty}\lambda^{\B{\varphi}}_m
    	   \rsp{\oB{w}}{\B{w}^{\B{\varphi}}_m}{\B{\varphi}}^2
    	\ge 
    	   \lambda^{\B{\varphi}}_j,
	\end{align*}
    as $\oB{w}\in \Hd$ is $\Lzrp$-normalized.
	Hence, it only remains to show that
	\begin{align*}
    	   \underset{k\to\infty}{\liminf\,}
    	       \lambda^{\B{\varphi}_k}_j
	   \ge 
	       \BE{\oB{w}}{\oB{w}}.
	\end{align*}
	This, however, can be proven completely analogously as in the proof of Lemma~\ref{lus}.
	
	In summary, we obtain the convergence
	\begin{align*}
	   \lambda_j^{\B{\varphi}_k}\to \lambda_j^{\B{\varphi}}, \quad \text{as $k\to\infty$},
	\end{align*}
	along a non-relabeled subsequence. Since the limit does not depend on any subsequence extraction, this convergence holds true for the whole sequence.
	As in Corollary~\ref{Wst}, we conclude that $\oB{u}_j:=\oB{w}\in \Hd$ is an eigenfunction to the eigenvalue $\lambda_j^\bphi$. In view of \eqref{wlk}, this completes the proof.
\end{proof}

\subsection{Local Lipschitz continuity of the eigenvalues}

The following lemma shows that all eigenvalues are locally Lipschitz continuous with respect to $\bphi$.

\begin{Lem}[Local Lipschitz continuity of the eigenvalues]\label{llip}
	Let $i\in \N$ be any index and let $\bphi\in \HL$ be arbitrary.
	Then there exist $\delta_i^\bphi,\,C^i_{\bphi}>0$ such that
	\begin{align*}
        	\left|
        	   \lambda^{\bphi}_i-\lambda^{\bphi+\bh}_i
        	\right|
    	\le 
    	   C^i_{\bphi}\norm{\bh}_{\HN\cap\LuN},
	\end{align*}
	for all $\bh\in \HN\cap \LuN$ with $\norm{\B h}_{\HL}<\delta_i^\bphi$.
	This means that the mapping
	\begin{align*}
	   \lambda_i: \HN\cap\LuN \to \R_{>0},\quad
	   \bphi\mapsto \lambda_i^\bphi,
	\end{align*}
	is locally Lipschitz continuous.
\end{Lem}
\begin{proof}
	Let $\bw_i\in \Hd$ denote a $\Lzrp$-normalized eigenfunction to the eigenvalue $\lambda^{\bphi}_i$. In the same fashion, let $\B{w}^{\bphi+\bh}_i\in \Hd$ denote a $\Lzrpa{\bphi+\bh}$-normalized eigenfunction to the eigenvalue $\lambda^{\bphi+\bh}_i$. Then, if $\delta_i^\bphi$ is sufficiently small, we obtain the estimate
	\begin{align*}
        	\left|
        	   \big(
        	       \lambda^{\bphi}_i-\lambda^{\bphi+\bh}_i
        	   \big)
        	   \brsp{\B{w}^{\bphi+\bh}_i}{\bw_i}{\varphi}
        	\right|
    	&\le
        	\left|
        	       \lambda_i^{\bphi}
        	       \brsp{\B{w}^{\bphi+\bh}_i}{\bw_i}{\varphi}
        	   -
        	       \lambda^{\bphi+\bh}_i
        	       \brsq{\B{w}^{\bphi+\bh}_i}{\bw_i}{\bphi+\bh}
        	\right|\\
    	&\quad
        	+
        	   \left|
            	       \lambda^{\bphi+\bh}_i
            	       \brsq{\B{w}^{\bphi+\bh}_i}{\bw_i}{\bphi+\bh}
        	       -
        	            \lambda^{\bphi+\bh}_i
        	             \brsq{\B{w}^{\bphi+\bh}_i}{\bw_i}{\bphi}
        	   \right|\\[1ex]
    	&=
        	\left|
            	\langle
            	   \mathcal{E}\big(\B{w}^{\bphi+\bh}_i\big),\mathcal{E}\big(\bw_i\big)
            	\rangle_{\C(\bphi)-\C(\bphi+\bh)}
        	\right|\\
    	&\quad
        	+
            	\left|
            	   \lambda^{\bphi+\bh}_i
            	   \Big(
            	           \brsq{\B{w}^{\bphi+\bh}_i}{\bw_i}{\bphi+\bh}
        	           -	
        	               \brsp{\B{w}^{\bphi+\bh}_i}{\bw_i}{\bphi}
        	       \Big)
        	   \right|\\[1ex]
    	&\le 
        	C^i_{\bphi}
        	\norm{\bh}_{\LuN},
	\end{align*}
	where the last inequality holds due to the local Lipschitz continuity of $\C$ and $\rho$, and the boundedness of $\lambda_i^{\bphi+\bh}$ which follows from Theorem~\ref{slw}. Note that the constant $C_i^{\bphi}$ may depend on $\lambda_i^{\bphi}$ but not on the eigenfunctions we have chosen, as they were assumed to be normalized.
	
	Suppose now that there exists a zero sequence $\left(\bh_k\right)_{k\in \N}\subset \HL$ such that
	\begin{align*}
        	\left|
        	   \lambda^{\bphi}_i-\lambda^{\bphi+\bh_k}_i
        	\right|
    	>
    	   k\norm{\bh_k}_{\HL},
	\end{align*}
	as $k\to \infty$.
	For the corresponding sequence of eigenfunctions $(\bw^{\bphi+\bh_k}_i)_{k\in \N}\subset \Hd$ for the eigenvalues $(\lambda_i^{\bphi+\bh_k})_{k\in\N}$, we know from Theorem~\ref{slw} that we find a $\Lzrp$-normalized eigenfunction $\oB{w}$ to the eigenvalue $\lambda_i^{\bphi}$ such that
	\begin{align*}
	   \bw^{\bphi+\bh_k}_i\to \oB{w}\quad \text{ in } \Lzrp,
	\end{align*}
	as $k\to \infty$, up to subsequence extraction.
	In particular, for $k$ sufficiently large, we know that the members of this subsequence satisfy
	\begin{align*}
	   \brsp{\bw_i^{\bphi+\bh_k}}{\oB{w}}{\varphi}>\frac{1}{2}
	\end{align*}
	and thus, 
	\begin{align*}
	   k\norm{\bh_k}_{\HL}<2C_{\bphi}^{i}\norm{\bh_k}_{\HL},
	\end{align*}
	which is an obvious contradiction. This proves the claim.
\end{proof}

\subsection{A sign convention for the eigenfunctions}

In the previous analysis there was no need to assume that the eigenspaces are one-dimensional. However, in Section~5, we want to show that the eigenvalues are Fr\'echet differentiable with respect to the phase-field. Therefore, it will be necessary to assume that for fixed $\bphi\in \HL$ the eigenspace corresponding to the considered eigenvalue $\lambda_i^{\bphi}$ is one-dimensional. In this case the eigenvalue is called \emph{simple}.

Simplicity of $\lambda_i^{\bphi}$ allows us to choose a corresponding eigenfunction $\bw_i^{\bphi}\in \Hd$ that is normalized with respect to the scalar product on $\Lzrp$ and \emph{unique} up to multiplication by $\pm 1$.
We call such an eigenfunction a \emph{representative} corresponding to $\lambda_i^{\bphi}$. 

In general, any eigenspace could be higher dimensional. For numerically motivated examples showing that even the simplicity of the first eigenvalue of a scalar elliptic regular PDE is no longer fulfilled in the vector valued case, see \cite{Beaudouin}.
However, in concrete applications, there are physical and numerical justifications for assuming simple eigenvalues.
This is due to the fact that nature as well as numerical simulations on computers lead to perturbations of the non-generic case of equal eigenvalues.

As a classical two dimensional example to illustrate this behavior, an eigenvalue problem associated with the Laplacian subject to Dirichlet boundary conditions can be considered.
If the domain is a perfect circle, eigenvalues with higher multiplicity will occur. 
However, as soon as the perfect circular shape of the domain is perturbated by small imperfections, these eigenvalues will become different and simple.
For more details see \cite{Barbarosie}.

In the following lemma, we will introduce a condition to fix a sequence of representatives whose elements $\bw_{i}^{\bphi_k}$ are uniquely determined if $\bphi_k\in \HL$ is sufficiently close to $\bphi$. 
In particular, we see that it is possible to deduce simplicity of the eigenvalues $\lambda^{\bphi_k}_i$ in a suitable neighborhood of $\lambda_i^{\bphi}$.

\begin{Lem} \label{LEM:SIGN}
	Let $i\in \N$ and $\left(\bphi_k\right)_{k\in \N}\subset \HL$ be a sequence such that
	\begin{align*}
	   \bphi_k\to \bphi \quad\text{in }\HL,
	\end{align*}
	for $k\to \infty$. Moreover, we assume that $\lambda^{\bphi}_i$ is a simple eigenvalue of \eqref{WEstate} and let $\bw_i^{\bphi}$ be a corresponding $\Lzrp$-normalized eigenfunction.
	
	Then for any $\eps\in (0,1)$, we can find a $K^{\eps}_i>0$ such that for any $k>K_i^\eps$, there exists a unique $\Lzrpa{\bphi_k}$-normalized eigenfunction $\bw_i^{\bphi_k}\in \Hd$ to the eigenvalue $\lambda_{i}^{\bphi_k}$ satisfying 
	\begin{align}\label{kombe}
	   \rsp{\bw_i^{\bphi_k}}{\bw_i^{\varphi}}{\bphi}>\eps.
	\end{align} 
	In particular, the eigenvalues $\lambda_i^{\bphi_k}$ with $k>K_i^\eps$ are simple.
\end{Lem}
\begin{proof}
	Note that we did not make any assumptions on the simplicity of the eigenspaces corresponding to $\lambda_{i}^{\bphi_k}$ for $k\in \N$. However, this can be established if $\bphi_k$ is close to $\bphi$ by invoking the simplicity of the eigenspace corresponding to $\lambda_i^{\bphi}$ and using the continuity properties known from Theorem~\ref{slw}.
		
	In the following, we will assume, without loss of generality, that $k$ is large enough to ensure that all eigenspaces to the eigenvalues $\lambda_{i}^{\bphi_k}$ are simple.
	If we are now able to find a sequence of representatives $\bw_i^{\bphi_k}$ that fulfills $\eqref{kombe}$ for a suitable $K^{\eps}_i\in \N$, then the uniqueness assertion is clear since the eigenfunctions are normalized and their sign is fixed by \eqref{kombe}.
	
	To prove the existence of such a sequence, we argue once more by contradiction. Let $\eps\in(0,1)$ be arbitrary and let us assume that there is no $K^{\eps}_i\in \N$ such that $\eqref{kombe}$ is fulfilled. Hence, after possibly swapping some of the signs, we can extract a subsequence such that 
	\begin{align}\label{contra}
        	\abs{
        		\rsq{\bw^{\bphi_k}_i}{\bw^{\bphi}_i}{\bphi}
        	       }
    	\le 
    	   \eps<1,
    	\quad\text{for all $k\in\N$}.
	\end{align}
	Using Theorem~\ref{slw} we obtain a weak limit $\oB{w}$ of a non-relabeled subsequence of $\left(\bw^{\bphi_k}_i\right)_{k\in \mathbb{N}}$ and infer from the simplicity of $\lambda_i^{\bphi}$ that $\oB{w}=\pm\bw_i^{\bphi}$. Hence, using \eqref{contra}, we obtain 
	\begin{align*}
    	   1
    	=
    	   \rsq{\bw^{\bphi}_i}{\bw^{\bphi}_i}{\bphi}
    	=
    	   \pm\rsq{\oB{w}}{\bw^{\bphi}_i}{\bphi}
    	<1,
	\end{align*}
	which is obviously a contradiction.
	
	Eventually, this means that condition $\eqref{kombe}$ allows us to pick a \emph{unique} representative $\bw_i^{\bphi_k} $ for every $k\in \N$ sufficiently large such that the obtained sequence fulfills
	\begin{align*}
	   \bw^{\bphi_k}_i\rightharpoonup \bw^{\bphi}_i \quad\text{in }\Hd,
	\end{align*}
	as $k\to \infty$, up to subsequence extraction.
\end{proof}

\medskip

The following corollary is a direct consequence of Lemma~\ref{LEM:SIGN}.

\begin{Cor} \label{COR:SIGN}
	For $i\in \N$ and $\bphi\in \HL$, we suppose that the eigenvalue $\lambda_i^\bphi$ is simple.
	Let $\bw_i^\bphi$ be a $L^2_{\bphi}(\Omega;\R^d)$-normalized eigenfunction to the eigenvalue $\lambda_i^\bphi$.
	
	Then, for all $\eps>0$, there exists $\delta>0$ such that for all
	\begin{align*}
	   \bh\in \LuN\cap \HN \quad\text{with}\quad \norm{\bh}_{\HN\cap \LuN}<\delta
	\end{align*}
	there exists a unique $L^2_{\bphi+\bh}(\Omega;\R^d)$-normalized eigenfunction $\bw_i^{\bphi+\bh}$ to the eigenvalue $\lambda_i^{\bphi+\bh}$ satisfying the condition 
	\begin{align}
	 \label{SC}
	   \brsp{\bw^{\bphi+\bh}_i}{\bw^\bphi_i}{\varphi}>\eps>0.
	\end{align}
	In particular, the eigenvalues $\lambda_i^{\bphi+\bh}$ are simple.
\end{Cor}

This means that, if $\bh$ is sufficiently small, the signs of the eigenfunctions $\bw^{\bphi+\bh}_i$ can be uniquely fixed in accordance with the sign of $\bw_i^\bphi$ by the sign condition \eqref{SC}.

\subsection{Continuity of the eigenfunctions}

In view of the sign convention from Corollary~\ref{COR:SIGN}, we can now prove the following continuity result.

\begin{Lem}[Continuity of eigenfunctions to simple eigenvalues]\label{dwko}
	Let $\bphi\in \HL$ be arbitrary and let $\bw_i^\bphi$ denote a $\Lzrp$-normalized eigenfunction to the eigenvalue $\lambda_i^\bphi$ which is assumed to be simple. 
	For any $\eps>0$, we assume that $\delta>0$, $\bh$ and the eigenfunctions $\bw_i^{\bphi+\bh}$ to the eigenvalues $\lambda_i^{\bphi+\bh}$ are all chosen in such a way that the sign condition \eqref{SC} is satisfied.
	
	Then the eigenfunctions $\bw_i^{\bphi+\bh}$ are uniquely determined and it holds that
	\begin{align}\label{Lkonv}
	   \big\|\bw^{\bphi+\B h}_i-\bw^\bphi_i \big\|_{\Hd}\to 0,
	\end{align} 
	as $\bh\to \B{0}$ in $\HN\cap\LuN$. This means that the mapping
	\begin{align*}
	   \bw_i : \HL \to \Hd,\quad \bphi \mapsto \bw^\bphi_i
	\end{align*}
	is (strongly sequentially) continuous with respect to the norm on $\Hd$.
\end{Lem}

\begin{proof}
	Let $(\bh_k)_{k\in\N} \subset \HL$ be any arbitrary sequence satisfying
	\begin{align}
	   \norm{\bh_k}_{\HL} < \delta \quad \text{for all $k\in\N$}.
	\end{align}
	Defining the sequence $(\bphi_k)_{k\in\N} \subset \HL$ by $\bphi_k := \bphi + \bh_k$ for all $k\in\N$, we can apply Theorem~\ref{slw} to conclude that
	\begin{align*}
	   \DB{w}{\bphi_k}_i\to \bw^\bphi_i \quad\text{in $\Lzd$ as $k\to\infty$,} 
	\end{align*} 
	along a non-relabeled subsequence. However, as the limit does not depend on the extracted subsequence, this convergence even holds true for the whole sequence. Note that for this reasoning it is essential that all members of the sequence are fixed by the sign convention \eqref{SC}. As the sequence $(\bh_k)_{k\in\N}$ was arbitrary, we further infer that
	\begin{align}\label{L2konv}
	   \bnorm{\DB{w}{\varphi+h}_i-\bw^\bphi_i}_{\Lzd}\to 0, 
	   \quad\text{as $\bh\to \B{0}$ in $\HN\cap\LuN$.}
	\end{align}
	If we can now show that
	\begin{align}
	\label{CONV:BE}
	   \BE{\bw^\bphi_i-\DB{w}{\varphi+h}_i}{\bw^\bphi_i-\DB{w}{\varphi+h}_i}\to 0,
	\end{align}
	as $\bh\to \B{0}$ in $\HN\cap\LuN$, the prove is completed.
	
	To this end, let $\bh\in\HL$ with  
	\begin{align*}
	   \norm{\bh}_{\HL} < \delta,
	\end{align*}
	be arbitrary. We derive the identity
	\begin{alignat*}{2}
	&\BE{\bw^\bphi_i-\DB{w}{\varphi+h}_i}{\bw^\bphi_i-\DB{w}{\varphi+h}_i}\\
    	&\quad=
        	\left[
        	       \BE{\bw^\bphi_i}{\bw^\bphi_i-\DB{w}{\varphi+h}_i}
        	   -\left\langle 		
        	       \mathcal{E}(\DB{w}{\varphi+h}_i),\mathcal{E}(\bw^\bphi_i-\DB{w}{\varphi+h}_i)
        	      \right\rangle_{\C(\bphi+\bh)}
        	\right]\\
    	&\qquad
        	+\left[
        	       \left\langle 
        	           \mathcal{E}(\DB{w}{\varphi+h}_i),\mathcal{E}(\bw^\bphi_i-\DB{w}{\varphi+h}_i)
        	       \right\rangle_{\C(\bphi+\bh)}
        	       -\BE{\DB{w}{\varphi+h}_i}{\bw^\bphi_i-\DB{w}{\varphi+h}_i}
        	\right].
	\end{alignat*}
	The second summand on the right-hand side converges to $0$ in $\HL$ as $\bh\to \B{0}$, since the norm $\|\DB{w}{\varphi+h}_i\|_{\Hd}$ is bounded by a constant that may depend on $\delta$ but not on $\bh$, and $\C$ is locally Lipschitz continuous. As $\bw_i^\bphi$ and $\bw_i^{\bphi+\bh}$ are eigenfunctions, they satisfy the state equation \eqref{WEstate} and we thus get
	\begin{alignat*}{2}
    	&
        	   \BE{\bw^\bphi_i}{\bw^\bphi_i-\DB{w}{\varphi+h}_i}
        	-
        	   \left\langle 
        	       \mathcal{E}\left(
                        \DB{w}{\varphi+h}_i\right),\mathcal{E}\left(\bw^\bphi_i-\DB{w}{\varphi+h}_i
                                    \right)
        	   \right\rangle_{\C(\bphi+\bh)}\\
    	&=
        	   \lambda^{\bphi}_i 
        	       \rsp{\bw^\bphi_i}{\bw^\bphi_i-\DB{w}{\varphi+h}_i}{\varphi}
        	-
        	   \lambda^{\bphi+\bh}_i
        	       \rsq{\DB{w}{\varphi+h}_i}{\bw^\bphi_i-\DB{w}{\varphi+h}_i}{\bphi+\bh}\\
    	&=
        	       \lambda^{\bphi}_i
        	           \left[ 	
        	                   \rsp{\bw^\bphi_i}{\bw^\bphi_i-\DB{w}{\varphi+h}_i}{\varphi}
        	               -
        	                   \rsp{\bw_i^{\bphi+\bh}}{\bw^\bphi_i-\DB{w}{\varphi+h}_i}{\varphi}
        	           \right]\\
    	&\quad+
        	   \lambda^{\bphi}_i
        	       \left[
        	               \rsp{\bw_i^{\bphi+\bh}}{\bw^\bphi_i-\DB{w}{\varphi+h}_i}{\varphi}
        	           -
        	               \rsq{\DB{w}{\varphi+h}_i}{\bw^\bphi_i-\DB{w}{\varphi+h}_i}{\bphi+\bh}
        	       \right]\\
    	&\quad
        	+
        	   \left[
        	       \lambda^{\bphi}_i-\lambda^{\bphi+\bh}_i
        	   \right]
        	   \rsq{\DB{w}{\varphi+h}_i}{\bw^\bphi_i-\DB{w}{\varphi+h}_i}{\bphi+\bh}.
	\end{alignat*}
	Here, the first summand converges to zero because of $\eqref{L2konv}$. The second summand converges to zero due to the local Lipschitz continuity of $\rho$ and the last summand converges to zero as a consequence of Theorem~\ref{slw}. This verifies \eqref{CONV:BE} and thus, the proof is complete.
\end{proof}


\section{Differentiability of the eigenvalues and the associated eigenfunctions}

\subsection{A formal consideration} \label{FC}

First of all, we want to discuss the desired differentiability results formally.
To obtain the Fr\'echet derivative of the functional
\begin{align*}
    \lambda_i: \HN\cap\LuN \to \R_{>0},\quad
    \bphi\mapsto \lambda_i^{\bphi},
\end{align*}
for $i\in \N$, we formally differentiate the state equation in the Gâteaux sense. 
If $\bw$ is an eigenfunction to the eigenvalue $\lambda_i^\bphi$, we have
\begin{align}
\label{EQ:VAR}
        \BE{\bw}{\bet}
    =	
        \lambda_i^{\bphi}\int_{\Omega}\rho(\bphi)\bw\cdot\bet\text{\,d}x
        \quad \text{for all $\bet\in\Hd$}.
\end{align}
 Computing the first variation of \eqref{EQ:VAR} with respect to $\bphi$ in the direction $\bh\in \HL$, and
choosing $\bw$ and $\bet$ as the $\Lzrp$-normalized eigenfunction $\DB{w}{\varphi}_i$ afterwards, we get
\begin{align}\label{ldif}
            \ld{\varphi}{i}\bh
    =
            \BED{\DB{w}{\varphi}_i}{\DB{w}{\varphi}_i}
        -
            \lambda_i^{\bphi}
            \int_{\Omega}\rho^{\prime}(\bphi)\bh\abs{\DB{w}{\varphi}_i}^2\text{\,d}x. 
\end{align}
Moreover, firstly plugging $\bw=\bw_i^\varphi$ into \eqref{EQ:VAR}, and then computing the first variation with respect to $\bphi$ in the direction $\bh\in \HL$ reveals 
that the formal derivative $\left(\bw^{\bphi}_i\right)^{\prime}\bh$ of $\bw^{\bphi}_i$ has to fulfill the equation
\begin{align}
    \begin{split}
        \begin{aligned}
            &
                    \BE{\left(\bw^{\bphi}_i\right)^{\prime}\bh}{\bet}
                -		
                    \lambda_i^{\bphi}
                    \int_{\Omega}\rho(\bphi)\left(\bw^{\bphi}_i\right)^{\prime}\bh\cdot\bet\text{\,d}x\\
            &\quad=		
                -
                    \BED{\bw^{\bphi}_i}{\bet}
                +
                    \lambda_i^{\bphi}
                    \int_{\Omega}\rho^{\prime}(\bphi)\bh\bw^{\bphi}_i\cdot \bet\text{\,d}x\\
            &\qquad
                +
                    \left(\lambda^{\bphi}_i\right)^{\prime}\bh
                    \int_{\Omega}\rho(\bphi)\bw^{\bphi}_i\cdot\bet\text{\,d}x.
        \end{aligned}
    \end{split}	
\end{align}
In the following, we intend to verify these results rigorously. We already see that formula $\eqref{ldif}$ is a priori not well-defined if there are at least two orthogonal eigenfunctions to the eigenvalue $\lambda_{i}^{\bphi}$, i.e., if $\lambda_{i}^{\bphi}$ is not simple. In the following approach, we will see that the simplicity of eigenvalues will play a crucial role in our analysis.

\subsection{Semi-differentiability of the first eigenvalue}

In \cite[Sect.~4.2]{Rousselet}, the concept of \emph{semi-differentiability} is introduced and applied to the first eigenvalue of an abstract problem discussed there. Semi-differentiability is a concept similar to Gateâux-differentiability but the limit does not need to fulfill any linearity or continuity assumptions, and the variation is only performed along a \emph{fixed positive direction}. The advantage becomes clear by the following example presented in \cite[Sect.~2.5]{Henrot}.
We consider the matrix-valued function
\begin{align*}
    A:\R\to\R^{2\times 2},\quad
    A(t)=
        \begin{pmatrix}
            1-t&0\\
             0&1+t
        \end{pmatrix},
\end{align*}
whose first eigenvalue 
\begin{align*}
    \lambda_1:\R\to\R,\quad \lambda_1^t =1-\abs{t}
\end{align*}
is not simple at $t=0$. Of course, $\lambda_1$ is not classically differentiable in $t=0$, but we still obtain a well defined limit 
\begin{align*}
    \underset{\substack{t\to 0\\ t>0}}{\lim}\;\frac{1-\abs{t}-1}{t}=-1.
\end{align*} 
This means we can still compute some sort of derivative in a fixed positive direction.

We now give a precise definition of semi-differentiability which can be found, e.g., in \cite[Def.~4.6]{Rousselet}.
\begin{Def}[Definition of semi-differentiability]
	Let $X,Y$ be Banach spaces and let $D\subseteq X$ be an open subset. Then the map $T:D\to Y$ is called \emph{semi-differentiable} at the point $x\in D$ if for all $h\in X$, there exists $y(x,h) \in Y$ such that
	\begin{align*}
	   \underset{\substack{t\to 0 \\ t>0}}{\lim}\; \frac{T(x+th)-T(x)}{t} = y(x,h).
	\end{align*}
	In this case we write $T'(x)h = y(x,h)$ to denote the semi-derivative of $T$ at the point $x$ with respect to the direction $h$.
\end{Def}

We want to show that in our problem the first eigenvalue also fulfills this weaker notion of differentiability, which will be enough to deduce first-order necessary optimality conditions for the first eigenvalue as we only want to derive convex combinations where $t>0$.

The advantage of this approach is that we do not have to assume simplicity of the first eigenvalue in order to obtain semi-differentiability, whereas as illustrated in Section~\ref{FC}, we need such simplicity assumptions to obtain classical differentiability.

The semi-differentiability of the first eigenvalue is established by the following lemma.

\begin{Thm}[Semi-differentiability of the first eigenvalue]\label{GDO}
	Let $\bphi,\bh\in \HL$ be arbitrary and let us define
	\begin{align}\label{GDef}
    	   (\lambda_1^{\bphi})^{\prime}\bh
    	\coloneqq
    	   \inf\left\{
    	       \BED{\bu}{\bu}-\ev\rspd{\bu}{\bu}
    	       \,\left|\,
    	           \begin{aligned}
    	               &\bu\in \Hd \text{ is an} \\
    	               &\text{eigenfunction to }\ev \\ 
    	               &\text{with } \norm{\bu}_{\Lzrp}=1 
    	           \end{aligned}
    	       \right\}
    	\right..
	\end{align}
	Then we have
	\begin{align*}
        	\underset{\substack{t\to 0\\ t>0}}{\lim}\;
        	\frac{\lambda_1^{\bphi+t\bh}-\ev}{t}
    	=
    	     \evd
	\end{align*}
	and thus, the eigenvalue $\lambda_1^\bphi$ is semi-differentiable with respect to $\bphi$.
\end{Thm}

\begin{proof}
	We first prove that the infimum in \eqref{GDef} is actually attained by a minimizer.
	To this end, let $\bu\in F_{\text{ad}}$ be arbitrary, where the feasible set is given as
	\begin{align*}
    	   F_{\text{ad}}
    	\coloneqq 
        	\left\{
        	   \bu\in \Hd 
        	   \left|\;
        	       \begin{aligned}
        	           &\bu \text{ is an eigenfunction to } \ev \\ 
        	           &\text{with } \norm{\bu}_{\Lzrp}=1
        	       \end{aligned}
        	   \right.
        	\right\}.
	\end{align*}
	By the differentiability and the local Lipschitz continuity of $\C$, we infer that there exists a constant $c_\bphi>0$ and $t_0>0$ such that for all $t<t_0$,
	\begin{align} \label{EST:INF}
    	\begin{aligned}
    	   &-c_{\bphi}
    	       \norm{\bh}_{\HL}\norm{\bu}_{\Hd}^2
    	       \\[1ex]
    	   &\qquad \le -\abs{
    		  \frac{\BET{\bu}{\bu}{\bphi+t\bh}-\BET{\bu}{\bu}{\bphi}}{t}}\\[1ex]
    	   &\qquad \le
    	       1 + \BED{\bu}{\bu}.
	   \end{aligned}
	\end{align}
	Using $\eqref{WEstate}$, we conclude that
	\begin{align}\label{uHb}
	   \norm{\bu}_{\Hd}^2\le \ev \quad\text{for all $\bu\in F_{\text{ad}}$}.
	\end{align}
	Moreover, due to \eqref{rhabs} and \eqref{SKP}, there exists a constant $c_{\bphi}^*>0$ such that
	\begin{align} \label{uurph}
    	   \abs{\rspd{\bu}{\bu}}
    	\le c_{\bphi}^*\,\norm{\bh}_{\HL}
    	 \quad\text{for all $\bu\in F_{\text{ad}}$}.
	\end{align}	
	Eventually, combining the above estimates, we conclude that
	\begin{align*}
	   &\BED{\bu}{\bu} -\ev\rspd{\bu}{\bu} \\
	   &\quad \ge - \left(c_\bphi + c_\bphi^*  \right) \ev \norm{\bh}_{\HL}   -1 \;>-\infty
	\end{align*}
	for all $u\in F_{\text{ad}}$. This directly implies that the infimum $(\lambda_1^{\bphi})^{\prime}\bh$ exists.
	
	Hence, we can find a minimizing sequence $\left(\bu_{n}\right)_{n\in \N}\subset F_{\text{ad}}$ such that 
	\begin{align*}
	   \underset{n\to \infty}{\lim} \left[\BED{\bu_n}{\bu_n}-\ev\rspd{\bu_n}{\bu_n}\right]=\evd.
	\end{align*}
	Due to $\eqref{uHb}$, there exists $\bu^{\ast}\in \Hd$ such that
	\begin{align*}
    	   \bu_n\rightharpoonup \bu^{\ast}\quad\text{in } \Hd,
    	\quad\text{and}\quad
    	   \bu_n\to \bu^{\ast}\quad\text{in } \Lzrp
	\end{align*}
	as $n\to \infty$, up to a subsequence extraction.
	In particular, this implies that $\bu^{\ast}\in F_{ad}$ which leads to
	\begin{align*}
	       \BE{\bu_n-\bu^{\ast}}{\bu_n-\bu^{\ast}}
       =
            \ev \rsp{\bu_n-\bu^{\ast}}{\bu_n-\bu^{\ast}}{\varphi}.
	\end{align*}
	This implies that $\bu_n\to \bu^{\ast}$ even strongly in $\Hd$. In particular we obtain
	\begin{align*}
    	   \evd
    	&=
            \underset{n\to \infty}{\lim} \left[\BED{\bu_n}{\bu_n}-\ev\rspd{\bu_n}{\bu_n}\right]\\
    	&=
            \BED{\bu^{\ast}}{\bu^{\ast}}-\ev\rspd{\bu^{\ast}}{\bu^{\ast}}.
	\end{align*}
	Hence, the infimum is attained at $\bu^*\in \Hd$.
	
	To prove 
	\begin{align}
     \label{ASS:SEMI}
    	\underset{\substack{t\to 0\\ t>0}}{\lim}\; \frac 1 t
    	   \abs{\lambda_1^{\bphi+t\bh}-\ev-(\ev)'[t\bh]}=0,
	\end{align}
	it suffices to show that there exist functions $f,g:\R\to \R$ with $f,g\in o(t)$ as $t\to 0$ such that for all $t>0$,
	\begin{align}
	   \lambda_1^{\bphi+t\bh}-\ev-(\ev)'[t\bh]&\le f(t)\label{do},\\
	   -\lambda_1^{\bphi+t\bh} +\ev + (\ev)'[t\bh]&\le g(t)\label{du}.
	\end{align}
	
	By the construction of $\bu^*$, we first observe that
	\begin{align*}
    	   &\lambda_1^{\bphi+t\bh}-\ev-(\ev)'[t\bh]\\
    	&\quad =
    	   \lambda_1^{\bphi+t\bh}
    	   -\BE{\bu^{\ast}}{\bu^{\ast}}
    	   -\BEDd{\bu^{\ast}}{\bu^{\ast}}
    	   +\ev\rspt{\bu^{\ast}}{\bu^{\ast}},
	\end{align*}
	since $\bu^{\ast}$ is an $\Lzrp$-normalized eigenfunction to the eigenvalue $\ev$. We compute
	\begin{align*}
    	&\lambda_1^{\bphi+t\bh}
        	-\BE{\bu^{\ast}}{\bu^{\ast}}
        	-\BEDd{\bu^{\ast}}{\bu^{\ast}}
        	+\ev\rspt{\bu^{\ast}}{\bu^{\ast}}\\
    	&\quad=
        	\left(
        	   \ev-\lambda_1^{\bphi+t\bh}
        	\right) 
        	\left(\bu^{\ast},\bu^{\ast}\right)_{
        		\rho(\bphi+t\bh)-\rho(\bphi)
        	   }
        	+\langle\E(\bu^{\ast}),\E(\bu^{\ast})\rangle_{
        		\C(\bphi+t\bh)
        		-\C(\bphi)
        		-\C^{\prime}(\bphi)t\bh
        	   }\\
    	&\qquad
        	-
            	\ev\left(\bu^{\ast},\bu^{\ast}\right)_{
            		\rho(\bphi+t\bh)
            		-\rho(\bphi)
            		-\rho^{\prime}(\bphi)t\bh
            	}
        	+
            	\lambda_1^{\bphi+t\bh}
            	\left(\bu^{\ast},\bu^{\ast}\right)_{\rho(\bphi+t\bh)}
            	-\langle\E(\bu^{\ast}),\E(\bu^{\ast})\rangle_{
            		\C(\bphi+t\bh)
        	}.
	\end{align*}
	Now the first three summands on the right-hand side are clearly in $o(t)$ as $t\to 0$ since the eigenvalues converge, and the functions $\rho$ and $\C$ are differentiable and locally Lipschitz continuous.
	For the remaining summands we can use the Courant--Fischer representation for the first eigenvalue which yields
	\begin{align*}
    	   \lambda_1^{\bphi+t\bh}
    	=
        	\min\left\{\left.
        	   \frac{\langle\E(\bu),\E(\bu)\rangle_{\C(\bphi+t\bh)}}{\norm{\bu}^2_{\Lzrpa{\bphi+t\bh}}}
        	\right|
        	   \begin{aligned}
        	       &\bu\in \Hd,\\
        	       &\bu\neq \B{0}
        	   \end{aligned}	
        	\right\}
        \le                     \frac{\langle\E(\bu^{\ast}),\E(\bu^{\ast})\rangle_{\C(\bphi+t\bh)}}{\norm{\bu^{\ast}}^2_{\Lzrpa{\bphi+t\bh}}}.
	\end{align*}
	This implies
	\begin{align*}
        	\lambda_1^{\bphi+t\bh}
        	 \left(\bu^{\ast},\bu^{\ast}\right)_{\rho(\bphi+t\bh)}
    	-
        	\langle\E(\bu^{\ast}),\E(\bu^{\ast})\rangle_{
        		\C(\bphi+t\bh)
        	}
        \le 0,
	\end{align*}
	and thus, $\eqref{do}$ is established.
	
	To prove $\eqref{du}$, we argue by contradiction and assume that \eqref{du} does not hold. Then, there exists $\eps>0$ and a sequence $(t_k)_{k\in\N}\subset (0,1]$ with $t_k\to 0$ as $k\to \infty$ such that for all $k\in\N$,
	\begin{align*}
	   -\lambda_1^{\bphi+t_k\bh} +\ev + (\ev)'[t_k\bh] \ge \eps t_k.
	\end{align*} 
	Then, according to Theorem~\ref{slw}, there exists a $L^2_\bphi(\Omega;\R^d)$-normalized eigenfunction $\bu$ to the eigenvalue $\lambda_1^\bphi$, as well as a sequence $(\bu^{\bphi+t_k\bh})_{k\in\N}$ consisting of $L^2_{\bphi+t_k\bh}(\Omega;\R^d)$-normalized eigenfunctions to the eigenvalues $(\lambda_1^{\bphi+t_k\bh})_{k\in\N}$ such that 
	\begin{align} \label{wHe}
    	   \bu^{\bphi+t_k\bh} \rightharpoonup \bu \quad\text{in $\Hd$}
    	\quad\text{and}\quad
    	   \bu^{\bphi+t_k\bh} \to \bu \quad\text{in $\Lzd$}
	\end{align}
	as $k\to\infty$, along a non-relabeled subsequence. Recalling the definition of $\evd$, we infer that
	\begin{align*}
    	&-\lambda_1^{\bphi+t_k\bh}+\ev+(\ev)'[t_k\bh] \\[1ex]
    	&\quad = -\lambda_1^{\bphi +t_k\bh}  (\bu^{\bphi +t_k\bh},\bu^{\bphi       +t_k\bh})_{\rho(\bphi+t_k\bh)} 
    	+ \ev + (\ev)'[t_k\bh] \\[1ex]
    	&\quad\le 
    	   -\bBEC{\bu^{\bphi+t_k\bh}}{\bu^{\bphi+t_k\bh}}{\C(\bphi+t_k\bh)}
    	   +\ev
    	   + \bBEC{\bu}{\bu}{\C'(\bphi)t_k\bh}\\
    	&\qquad\quad -\ev\left(\bu,\bu\right)_{\rho^{\prime}(\bphi)t_k\bh}.
	\end{align*}
	Recalling the identities
	\begin{align*}
    	   \bBE{\bu}{\bu^{\bphi+t_k\bh}} 
    	&= \lambda_1^\bphi \brsq{\bu}{\bu^{\bphi+t_k\bh}}{\bphi},\\
    	   \bBEC{\bu}{\bu-\bu^{\bphi+t_k\bh}}{\C(\bphi+t_k\bh)} 
    	&= \lambda_1^{\bphi+t_k\bh}  \brsq{\bu}{\bu-\bu^{\bphi+t_k\bh}}{\bphi+t_k\bh} ,
	\end{align*}
	a straightforward computation reveals that
	\begin{align*}
    	&-\lambda_1^{\bphi+t_k\bh}+\ev+(\ev)'[t_k\bh] \\[1ex]
    	&\quad\le 
    	   \big\langle\E(\bu),\E(\bu)\big\rangle_{
    		  \C^{\prime}(\bphi)t_k\bh
    		  +	\C(\bphi)
    		  -	\C(\bphi+t_k\bh)
    	}
    	   +
    	   \ev\big(\bu,\bu\big)_{
    		  \rho(\bphi+t_k\bh)
    		      -	\rho(\bphi)
    		      -	\rho^{\prime}(\bphi)t_k\bh
    	   }\\
    	&\qquad
    	   +
    	    \big\langle\E(\bu),\E(\bu-\bu^{\bphi+t_k\bh})\big\rangle_{
    		  \C(\bphi+t_k\bh)
    		      -	\C(\bphi)
    		      -	\C^{\prime}(\bphi)t_k\bh
    	       }\\
    	&\qquad
    	   +
    	   \big\langle\E(\bu),\E(\bu-\bu^{\bphi+t_k\bh})\big\rangle_{
    		  \C^{\prime}(\bphi)t_k\bh
    	       }
    	   + 
    	   \lambda_1^{\bphi+t_k\bh}
    	   \big(\bu,\bu-\bu^{\bphi+t_k\bh}\big)_{
    		  \rho(\bphi)-\rho(\bphi+t_k\bh)
    	   }
    	\\
    	&\qquad
    	   +(\ev-\lambda_1^{\bphi+t_k\bh})
    	       \big(\bu,\bu-\bu^{\bphi+t_k\bh}\big)_{\rho(\bphi)}
    	   +
    	       \big(\lambda_1^{\bphi+t_k\bh}-\ev\big)
    	       \big(\bu,\bu\big)_{\rho(\bphi+t_k\bh)-\rho(\bphi)}.
	\end{align*}
	Recalling the convergence property $\eqref{wHe}$, that $\rho$ and $\C$ are of class $C^1_\text{loc}$, that $\bphi\mapsto\lambda_1^\bphi$ is locally Lipschitz continuous according to Lemma~\ref{llip}, and that
	\begin{align*}
	   \BED{\bu}{\cdot}\in \left(\Hd\right)^{\ast},
	\end{align*}
	we conclude that the right-hand side belongs to $o(t_k)$ as $k\to\infty$.
	
	On the other hand we assumed
	\begin{align*}
	   \eps t_k \le -\lambda_1^{\bphi+t_k\bh} +\ev + (\ev)'[t_k\bh],
	\end{align*}
	which is obviously a contradiction as the inequality cannot hold for $k$ sufficiently large. This proves \eqref{du}.
	
	Now, \eqref{ASS:SEMI} directly follows from \eqref{do} and \eqref{du} and thus, the proof is complete.
\end{proof}

\medskip

\subsection{Fr\'echet differentiability of eigenvalues and their corresponding eigenfunctions}

If the considered eigenvalue is simple, we can even obtain stronger differentiability results in the Fr\'echet sense.
To be precise, if for $i\in\N$ and  $\bphi\in \HN\cap\LuN$, the eigenvalue $\lambda_i^\bphi$ associated with $\bphi$ is simple, then $\lambda_i^\bphi$ and any fixed $L^2_\bphi(\Omega;\R^d)$-normalized eigenfunction $\bw_i^\bphi$ are even Fr\'echet-differentiable with respect to $\bphi$. This is established by the following theorem:

\begin{Thm}[Fr\'echet-differentiability of simple eigenvalues and their eigenfunctions]\label{difewv}
	Let $\bphi\in \HN\cap\LuN$ be arbitrary and suppose that for $i\in\N$, the eigenvalue $\lambda_i^\bphi$ is simple. We further fix a $L^2_\bphi(\Omega;\R^d)$-normalized eigenfunction $\bw_i^\bphi$ to the eigenvalue $\lambda_i^\bphi$.
	
	Then there exist constants $\delta_i^\bphi,r_i^\bphi>0$ such that the operator 
	\begin{align*}
        	S^{\bphi}_i: 
        	B_{\delta_i^\bphi}(\bphi)\subset \HN\cap\LuN 
    	&\to 
        	B_{r_i^\bphi}
        	   \big(
        	       (\DB{w}{\varphi}_i,\lambda_i^{\bphi})
        	   \big)
        	\subset\Hd\times \R,\\
        	\btheta
    	&\mapsto 
    	   \big(\bw^{\btheta}_i,\lambda^{\btheta}_i\big),
	\end{align*}
	is well-defined and continuously Fr\'echet differentiable. Here, $\bw_i^\btheta$ denotes the unique $L^2_\btheta(\Omega;\R^d)$-normalized eigenfunction to the eigenvalue $\lambda_i^\btheta$ satisfying the sign condition \eqref{SC} written for $\bh=\btheta-\bphi$.
	
	Moreover, for any $\bh\in \HN\cap\LuN$, the Fr\'echet derivative $\left(\lambda^{\bphi}_i\right)^{\prime}\bh$ of the eigenvalue $\lambda^{\bphi}_i$ at $\bphi$ in the direction $\bh$ reads as
	\begin{align}\label{lhdef}
    	   \left(\lambda^{\bphi}_i\right)^{\prime}\bh
    	\coloneqq 	
    	   \left(
    	       S^{\bphi}_{i,2}(\bphi)
    	    \right)^{\prime}\bh
    	=
    	       \BED{\bw^{\bphi}_i}{\bw^{\bphi}_i}
    	   -
    	       \lambda^{\bphi}_i
    	        \int_{\Omega}\rho^{\prime}(\bphi)\bh\abs{\bw^{\bphi}_i}^2\textup{\,d}x,
	\end{align}
	and the Fr\'echet derivative $\left(\bw^{\bphi}_i\right)^{\prime}\bh\coloneqq \big(S^{\bphi}_{i,1}(\bphi)\big)^{\prime}\bh\in \Hd$ of the corresponding eigenfunction $\bw_i^{\bphi}$ at $\bphi$ in the direction $\bh$ is the unique solution of
	\begin{alignat}{2}\label{whdef}
    	\begin{aligned}
    	   &\BE{
    		  \left(\bw^{\bphi}_i\right)^{\prime}\bh
    	       }
    	       {\bet}
    	   -
    	       \lambda^{\bphi}_i
    	       \int_{\Omega}\rho(\bphi)\left(\bw^{\bphi}_i\right)^{\prime}\bh\cdot\bet\textup{\,d}x\\
    	&\quad=		
    	   -\BED{\bw^{\varphi}_i}{\bet}
    	   +
    	   \lambda^{\bphi}_i
    	       \int_{\Omega}\rho^{\prime}(\bphi)\bh\bw^{\varphi}_i\cdot\bet\textup{\,d}x\\
    	&\qquad
    	   +
    	   \left(\lambda^{\bphi}_i\right)^{\prime}\bh
    	   \int_{\Omega}\rho(\bphi)\bw^{\bphi}_i\cdot\bet\textup{\,d}x,
    	\end{aligned}
	\end{alignat}
	for all $\bet\in \Hd$, that fulfills
	\begin{align}\label{zB}
	   \rsp{\left(\bw^{\bphi}_i\right)^{\prime}\bh}{\bw_i^{\bphi}}{\bphi}=
	   -\frac{1}{2}\int_{\Omega}\rho^{\prime}(\bphi)\abs{\bw_i^{\bphi}}^2\text{d}x.
	\end{align}
\end{Thm}
To prove of this theorem, we intend to apply the implicit function theorem (see, e.g., \cite[Theorem~4.B]{Zeidler1}). Therefore,  it is essential to show a bijectivity condition. In our setting this condition will be fulfilled if a certain PDE resulting from the eigenvalue equations has a unique solution. To show this existence and uniqueness we need to apply the Fredholm alternative established by Lemma~\ref{Freset}.

In the following, we use the space
\begin{align*}
    \HD\coloneqq \big(\Hd\big)^{\ast}
\end{align*}
along with the canonical embedding
\begin{align}\label{Kein}
    \Lzrp \hookrightarrow H^{-1}(\Omega,\R^d),\quad
    \B{v}\mapsto \left(\bet\mapsto \rsp{\B{v}}{\bet}{\bphi}\right).
\end{align}
In particular, the duality pairing is given by
\begin{align*}
    \DP{\cdot}{\cdot} = \rsp{\cdot}{\cdot}{\bphi}.
\end{align*}

\begin{Lem}[Fredholm alternative for the eigenvalue problem]\label{Freset}
	Let $\bphi\in \HN\cap\LuN$ be arbitrary and suppose that for $i\in\N$, the eigenvalue $\lambda_i^\bphi$ is simple. We further fix a $L^2_\bphi(\Omega;\R^d)$-normalized eigenfunction $\bw_i^\bphi$ to the eigenvalue $\lambda_i^\bphi$.
	
	Then there exists a solution $\bu\in \Hd$ of the equation
	\begin{align}\label{wdg}
    	   \BE{\bu}{\bet}-\lambda^{\bphi}_i\rsp{\bu}{\bet}{\varphi}
    	&= 
    	   \langle\B{f},\bet \rangle_{H^{-1},H^1},
	\end{align}
	for all $\bet\in \Hd$ if and only if $\B{f}\in \HD$ fulfills 
	\begin{align*}
	   \langle\B{f}, \bw^{\bphi}_i\rangle_{H^{-1},H^1}=0.
	\end{align*}
	In this case, there is a unique solution $\bu^{\perp}$ in $\Hd\cap \langle\bw^{\bphi}_i\rangle_{\textup{span}}^{\perp,\Lzrp} $, and any other solution can be expressed as $\bu^{\perp}+\alpha\bw^{\bphi}_i$
	for some $\alpha\in \R$.
\end{Lem}

\begin{proof}
	Suppose that $\B{f}\in \HD$. Then, $\bu\in \Hd$ is a solution of $\eqref{wdg}$ if and only if 
	\begin{align*}
	   \BE{\bu}{\bet}=\DP{\lambda^{\bphi}_i\bu+\B{f}}{\bet},
	\end{align*}
	for all $\bet\in \Hd$. 	
	As
	\begin{align*}
	   \BE{\cdot}{\cdot}:\Hd\times \Hd\to \R,
	\end{align*}
	is a continuous, coercive bilinear form, we are able to define a continuous and compact operator 
	\begin{align*}
	   L^{-1}:\HD\to \Hd\hookrightarrow\HD,\quad \B{g}\mapsto \DB{u}{g},
	\end{align*}
	that maps any right-hand side $\B{g}\in \HD$ onto its unique solution $\DB{u}{g}\in \Hd$ of 
	\begin{align}\label{Ldef}
	   \BE{\DB{u}{g}}{\bet}=\DP{\B{g}}{\bet},
	\end{align}
	for all $\bet\in \Hd$.
	
	In the following we understand $\HD$ as a Hilbert space endowed with the scalar product 
	\begin{align*}
	   \big(\B{g},\bh\big)_{L^{-1}}\coloneqq \bBE{L^{-1}\B{g}}{L^{-1}\bh}.
	\end{align*}
	Indeed, this defines a scalar product as $L^{-1}$ is injective. Note that due to \eqref{Ldef} and the the fact that $\BE{\cdot}{\cdot}$ is a scalar product on $\Hd$, the norm induced by $(\cdot,\cdot)_{L^{-1}}$ is equivalent to the canonical operator norm on $\HD$, which guarantees completeness of this space with respect to this new scalar product.
	
	In the following, we write $R(\cdot)$ and $N(\cdot)$ denote the range and the null space of a linear operator, respectively.
	It is easy to see that $L^{-1}$ is self-adjoint with respect to this scalar product. Furthermore the following equivalences are follow by a straightforward computation: 
	\begin{align} \label{eqn0}
    	\begin{aligned}
        	&\exists \bu\in \Hd \;\text{that solves}\; \eqref{wdg},\\
        	&\quad\Leftrightarrow\quad
        	   \exists \bu\in \Hd:\;\;
        	   \bu-\lambda^{\bphi}_i L^{-1}\bu=L^{-1}\B{f},\\
        	&\quad\Leftrightarrow\quad
        	   L^{-1}\B{f} \in R(Id-\lambda^{\bphi}_i L^{-1}).
    	\end{aligned}
	\end{align}
	Since $L^{-1}$ is compact, we have that
	\begin{align*}
	   Id-\lambda^{\bphi}_i L^{-1}: \HD\to\HD,
	\end{align*}
	is a Fredholm operator. In particular, we thus know that
	\begin{align*}
	   R(Id-\lambda^{\bphi}_i L^{-1})\subset \HD,
	\end{align*}
	is closed and 
	\begin{align*}
    	   R(Id-\lambda^{\bphi}_i L^{-1})
    	=
	       N(Id-\lambda^{\bphi}_i L^{-1})^{\perp,\HD}.
	\end{align*}
	Since $L^{-1}$ is self-adjoint, we infer that
	\begin{align} \label{eqn1}
    	\begin{aligned}
        	   L^{-1}\B{f}\in R(Id-\lambda^{\bphi}_i L^{-1})
        	&\quad\Leftrightarrow\quad
        	   \forall \B{v}\in N(Id-\lambda^{\bphi}_i L^{-1}):\; \bLP{L^{-1}\B{f}}{\B{v}}=0 ,\\
        	&\quad\Leftrightarrow\quad
        	   \forall \B{v}\in N(Id-\lambda^{\bphi}_i L^{-1}):\; \bLP{\B{f}}{L^{-1}\B{v}}=0 .
    	\end{aligned}
	\end{align}
	It further holds that
	\begin{align*}
    	\begin{aligned}
        	   & \B{v}\in N(Id-\lambda^{\bphi}_i  L^{-1}) \\
        	   &\quad\Leftrightarrow\quad \forall \bet \in\Hd:\;\;
        	   \lambda^{\bphi}_i\bDP{L^{-1}\B{v}}{\bet}
    	   =
    	       \bDP{\B{v}}{\bet}\\
    	&\quad\Leftrightarrow\quad \forall \bet \in\Hd:\;\;
    	       \lambda^{\bphi}_i\brsp{L^{-1}\B{v}}{\bet}{\bphi}
    	   =
    	       \bBE{L^{-1}\B{v}}{\bet} \\
    	&\quad\Leftrightarrow\quad 
    	   \text{$L^{-1}\B{v}\in \Hd$ is an eigenfunction to the eigenvalue $\lambda_i^{\bphi}$}\\
    	&\quad\Leftrightarrow\quad
    	       L^{-1}\B{v}
    	   \in 
    	       \langle
    	           \bw^{\bphi}_i
    	        \rangle_{\text{span}}\subset\Lzrp,
    	\end{aligned}
	\end{align*}
	where the last equivalence holds since $\lambda_i^{\bphi}$ was assumed to be simple and therefore, the corresponding eigenspace is one-dimensional. In view of \eqref{eqn1}, this means that
	\begin{align} \label{eqn2}
    	\begin{aligned}
        	   L^{-1}\B{f}\in R(Id-\lambda^{\bphi}_i L^{-1})
        	&\quad\Leftrightarrow\quad
        	   \forall \B{v}\in N(Id-\lambda^{\bphi}_i L^{-1}):\; \bLP{\B{f}}{L^{-1}\B{v}}=0 \\
        	&\quad\Leftrightarrow\quad
        	   \LP{\B{f}}{\bw^{\bphi}_i} = 0.
    	\end{aligned}	
	\end{align} 
	We further know that 
	\begin{align*}
	   L^{-1}\bw^{\bphi}_i= \frac{1}{\lambda^{\bphi}_i} \bw^{\bphi}_i\in \Hd.
	\end{align*}
	Hence, since $L^{-1}\B{f}$ is a solution of \eqref{Ldef}, we have
	\begin{align} \label{eqn3}
    	\begin{aligned}
    	       \LP{\B{f}}{\bw^{\bphi}_i}
    	   &=
    	       \BE{L^{-1}\B{f}}{L^{-1}\bw^{\bphi}_i}\\
    	   &=
    	       \DP{\B{f}}{L^{-1}\bw^{\bphi}_i} 
    	    =
    	       \frac{1}{\lambda^{\bphi}_i}
    	       \DP{\B{f}}{\bw^{\bphi}_i}.
    	\end{aligned}
	\end{align}
	Combining \eqref{eqn0}, \eqref{eqn2} and \eqref{eqn3}, we conclude that
	\begin{align*}
    	&\exists\, \bu\in \Hd \;\text{that solves}\; \eqref{wdg} \\
    	&\quad\Leftrightarrow\quad	
    	   L^{-1}\B{f}\in R(Id-\lambda^{\bphi}_i L^{-1}) \\
    	&\quad\Leftrightarrow\quad 
    	   \DP{\B{f}}{\bw^{\bphi}_i}=0.
	\end{align*}
	This proves the first assertion.
	
	Let us now assume that $\DP{\B{f}}{\bw^{\bphi}_i}=0$ and let
	\begin{align*}
    	P^{\bphi}_i: 
    	   \Lzrp 
    	\to 
    	   \langle	\bw^{\bphi}_i \rangle_{\text{span}}\subset \Lzrp,
	\end{align*}
	denote the orthogonal projection onto the linear subspace $\langle \bw^{\bphi}_i \rangle_{\text{span}}$ with respect to the scalar product on $\Lzrp$. For any solution $\bu\in \Hd$ of $\eqref{wdg}$ we obtain from the decomposition $\bu=\left(\bu-P^{\bphi}_i(\bu)\right)+P^{\bphi}_i(\bu)$ that
	\begin{align*}
    	   \langle\B{f},\bet \rangle_{H^{-1},H^1}
    	&=
        	   \bBE{\bu-P^{\bphi}_i(\bu)}{\bet}
        	-
        	   \lambda^{\bphi}_i
        	   \brsp{\bu-P^{\bphi}_i(\bu)}{\bet}{\varphi}.
	\end{align*}
	Hence, it also holds that 
	\begin{align*}
    	   \bu^{\perp}
    	\coloneqq 
    	   \bu-P^{\bphi}_i(\bu)
    	\in 
    	   \Hd\cap\langle
    	       \bw^{\bphi}_i 
    	   \rangle_{\text{span}}^{\perp,\Lzrp}
	\end{align*}
	fulfills equality $\eqref{wdg}$. Uniqueness of the solution $\bu^\perp$ follows from the simplicity of $\lambda_{i}^{\bphi}$ and the linearity of equation $\eqref{wdg}$. In particular, any solution $\bu\in\Hd$ can be expressed as
	\begin{align*}
	   \bu = \bu^\perp + P^{\bphi}_i(\bu) = \bu^\perp + \alpha \bw_i^\bphi
	\end{align*}
	for some $\alpha\in\R$. This completes the proof.
\end{proof}

\pagebreak[2]

We can now use the Fredholm alternative to prove Theorem~\ref{difewv}.
\begin{proof}[Proof of Theorem~\ref{difewv}]
	As mentioned above, we intend to apply the implicit function theorem to prove the assertion. To this end, we define the operator
	\begin{align*}
    	F: 
    	   \big(\HN\cap\LuN\big) \times \Hd \times \R
    	&\to 
    	   \HD\times\R,\\
    	   (\btheta,\bw,\lambda)
    	&\mapsto
    	   \begin{pmatrix}
    	       -\nabla\cdot \C(\btheta)\E(\bw)
    	       -\lambda \rho(\btheta)\bw\\
    	       \rsp{\bw}{\bw}{\vartheta}-1
    	\end{pmatrix}.
	\end{align*}
	Here we canonically understand the first component of the right-hand side as an element of $\HD$, i.e.,
	\begin{align*}
        	\langle 
        	   F_1(\btheta,\bw,\lambda),\bet
        	\rangle_{H^{-1},H_D^1}
    	=
    	     \BET{\bw}{\bet}{\btheta}
    	   -
    	       \lambda\rsp{\bw}{\bet}{\vartheta},
	\end{align*}
	for all $\bet\in \Hd$. First of all, it is clear that $\bphi\in \HN\cap\LuN$, the eigenvalue $\lambda_i^{\bphi}\in \R$ and the corresponding representative $\bw_i^{\bphi}\in \Hd$ satisfy
	\begin{align*}
	   F\left(\bphi,\bw^{\bphi}_i,\lambda^{\bphi}_i\right)=\B{0}.
	\end{align*}	
	To apply the implicit function theorem, we need to show that $F$ is of class $C^1$ on a suitable neighbourhood of $\left(\bphi,\bw^{\bphi}_i,\lambda^{\bphi}_i\right)$. For this purpose, we show that all partial Fr\'echet derivatives are continuous at any point in the domain of definition of $F$. 
	Formally computing the partial derivatives
	at a point
	\begin{align*}
	   (\btheta,\bw,\lambda)\in\big(\HN\cap\LuN\big)\times \Hd\times \R,
	\end{align*}
	in the direction
	\begin{align*}
	   (\bh,\bu,\mu)\in \big(\HN\cap\LuN\big)\times \Hd\times \R,
	\end{align*}
	gives
	\begin{subequations}
        \begin{align}
    	       \partial_{\btheta}F_1(\btheta,\bw,\lambda)\bh
    	   &=
            	-\nabla\cdot \C^{\prime}(\btheta)\bh\E(\bw)
            	-\lambda \rho^{\prime}(\btheta)\bh\bw,\label{darpa}\\
            	\partial_{\bw}F_1(\btheta,\bw,\lambda)\bu
    	   &=
            	-\nabla\cdot \C(\btheta)\E(\bu)
            	-\lambda \rho(\btheta)\bu,\label{darpb}\\
            	\partial_{\lambda}F_1(\btheta,\bw,\lambda)\mu
    	   &=
            	-\mu \rho(\btheta)\bw,\label{darpc}\\
            	\partial_{\btheta}F_2(\btheta,\bw,\lambda)\bh
    	   &=
            	\int_{\Omega}\rho^{\prime}(\btheta)\bh\abs{\bw}^2\text{\,d}x,\label{darpd}\\
            	\partial_{\bw}F_2(\btheta,\bw,\lambda)\bu
    	   &=
            	2\int_{\Omega}\rho(\btheta)\bw\cdot\bu\text{\,d}x,\label{darpe}\\
            	\partial_{\lambda}F_2(\btheta,\bw,\lambda)\mu
    	   &=0,\label{darpf}
       \end{align}
	\end{subequations}
    where the first two identities are to be understood in a weak sense.
	We can rigorously prove that the above expressions are actually the partial Fr\'echet derivatives. 
	Here, we will present a detailed proof only for \eqref{darpa} as all other derivatives can be verified analogously.
	We first notice that for any fixed $(\btheta,\bw,\lambda)\in \HN\cap\LuN\times \Hd\times \R$, it holds that
	\begin{align}\label{partL}
    	\begin{aligned}
        	   \Big[\bh \mapsto
        	       \partial_{\btheta}F_1(\btheta,\bw,\lambda)\bh
        	   \Big] 
    	   \in 
    	       \mathcal{L}\left(\HN\cap\LuN, \HD\right).
    	\end{aligned}
	\end{align}
	Indeed, the linearity of the above mapping is clear and the assumptions on $\C,\rho$ along with Hölder's inequality imply the existence of a constant $C>0$ such that
	\begin{align*}
    	&\norm{
    		  \partial_{\btheta}F_1(\btheta,\bw,\lambda)
    	    }_{\mathcal{L}\left(H^1\cap L^{\infty}, H^{-1}\right)} \\
    	&\quad=
    	   \underset{\norm{\bh}_{H^1\cap L^{\infty}}=1}{\sup}\;\;
    	   \underset{\norm{\bet}_{H^1_D}=1}{\sup}\;
    	       \abs{	\left\langle
    		      \E(\bw),\E(\bet)
    		      \right\rangle_{\C^{\prime}(\btheta)\bh}
    		  +
    		      \lambda\int_{\Omega}\rho^{\prime}(\btheta)\bh\bw\cdot\bet\text{\,d}x
    	           }\\
    	&\quad\le 
    	   C\norm{\bw}_{\Hd}.
	\end{align*}
	It further holds that
	\begin{align*}
    	&
        	\norm{
        		F_1(\btheta+\bh,\bw,\lambda)
        		-F_1(\btheta,\bw,\lambda)
        		-\partial_{\btheta}F_1(\btheta,\bw,\lambda)\bh
        	}_{\HD}\\
    	&\quad\le
        	\underset{\norm{\bet}_{\Hd}=1}{\sup}
        	\abs{
        		\langle 	
        		\E(\bw),\E(\bet)
        		\rangle_{\C(\btheta+\bh)}
        		-\BET{\bw}{\bet}{\btheta}
        		-\langle 
        		\E(\bw),\E(\bet)
        		\rangle_{\C^{\prime}(\btheta)\bh}
        	}\\
    	&\qquad +
        	\underset{\norm{\bet}_{\Hd}=1}{\sup}
        	\abs{\lambda}
        	\int_{\Omega}\abs{\rho(\btheta+\bh)-\rho(\btheta)-\rho^{\prime}(\btheta)\bh}\abs{\bw\cdot \bet}\text{\,d}x.
	\end{align*}
	Now, proceeding similarly as in \cite[Proof of Thm~3.3]{Blank}, we invoke the differentiability properties of $\C$ and $\rho$ to conclude that
	\begin{align*}
        	\norm{
        		\C(\btheta+\bh)
        		-\C(\btheta)
        		-\C^{\prime}(\btheta)\bh
        	}_{\LuN}
    	&\in 
        	o\left(\norm{\bh}_{\LuN}\right),\\
        	\norm{
        		\rho(\btheta+\bh)
        		-\rho(\btheta)
        		-\rho^{\prime}(\btheta)\bh}_{\LuN}
    	&\in 
    	   o\left(\norm{\bh}_{\LuN}\right).
	\end{align*}
	Hence, Hölder's inequality yields
	\begin{align*}
        	\norm{
        		F_1(\btheta+\bh,\bw,\lambda)
        		-F_1(\btheta,\bw,\lambda)
        		-\partial_{\btheta}F_1(\btheta,\bw,\lambda)\bh
        	}_{\HD}
    	\in 
    	   o\left(\norm{\bh}_{\HN\cap\LuN}\right),
	\end{align*}
	which proves that $\partial_{\btheta}F_1(\btheta,\bw,\lambda)$ is indeed the partial derivative of $F_1$ with respect to $\btheta$ in the Fr\'echet sense.
	
	It remains to prove the continuity of the partial Fr\'echet derivatives. Here we also present the proof only for \eqref{darpa} as the continuity of the other partial derivatives can be established similarly. Let $(\btheta_n,\bw_n,\lambda_n)_{n\in\N}$ denote any sequence in $\big(\HN\cap\LuN\big)\times \Hd\times \R$ satisfying
	\begin{align*}
    	   \left(\btheta_n,\bw_n,\lambda_n\right)
    	\to 
    	   \left(\btheta,\bw,\lambda\right)
    	\quad\text{in }
    	   \big(\HN\cap\LuN\big)\times \Hd\times \R,
	\end{align*}
	for $n\to \infty$. Then it holds that
	\begin{align*}
    	&
        	\underset{\norm{\bh}_{H^1\cap L^{\infty}}=1}{\sup}
        	\norm{
        		\partial_{\btheta}F_1(\btheta_n,\bw_n,\lambda_n)\bh
        		-\partial_{\btheta}F_1(\btheta,\bw,\lambda)\bh
        	}_{\HD}
        	\\[1ex]
    	&\quad \le 
        	\underset{\norm{\bh}_{H^1\cap L^{\infty}}=1}{\sup}\;\;
        	\underset{\norm{\bet}_{H^1_D}=1}{\sup}\;
        	\abs{
        		\langle 
        		\E(\bw_n),\E(\bet)
        		\rangle_{\C^{\prime}(\btheta_n)\bh}
        		-\langle
        		\E(\bw),\E(\bet)
        		\rangle_{\C^{\prime}(\btheta)\bh}
        	}
        	\\
    	&\qquad +
        	\underset{\norm{\bh}_{H^1\cap L^{\infty}}=1}{\sup}\;\;
        	\underset{\norm{\bet}_{H^1_D}=1}{\sup}\;
        	\abs{\int_{\Omega}
        		\lambda_n\rho^{\prime}(\btheta_n)\bh\bw_n
        		\cdot
        		\bet\text{\,d}x
        		-\int_{\Omega}
        		\lambda\rho^{\prime}(\btheta)\bh\bw
        		\cdot
        		\bet\text{\,d}x
        	}.
	\end{align*}
	Applying Hölder's inequality, and recalling the differentiability of $\C$, we infer that
	\begin{align*}
    	&
        	\abs{\langle 
        		\E(\bw_n),\E(\bet)
        		\rangle_{\C^{\prime}(\btheta_n)\bh}
        		-\langle
        		\E(\bw),\E(\bet)
        		\rangle_{\C^{\prime}(\btheta)\bh}
        	}\\[1ex]
    	&\quad\le
        	\abs{\langle 	
        		\E(\bw_n),\E(\bet)
        		\rangle_{\C^{\prime}(\btheta_n)\bh}
        		-\langle
        		\E(\bw_n),\E(\bet)
        		\rangle_{\C^{\prime}(\btheta)\bh}
        	}\\
    	&\qquad
        	+\abs{\langle 
        		\E(\bw_n),\E(\bet)
        		\rangle_{\C^{\prime}(\btheta)\bh}
        		-\langle
        		      \E(\bw),\E(\bet)
        		\rangle_{\C^{\prime}(\btheta)\bh}}\\[1ex]
    	&\quad\le 
        		C\left(
        		  \norm{\btheta_n-\btheta}_{H^1\cap L^{\infty}}\norm{\bw_n}_{\Hd}
        		+\norm{\bw_n-\bw}_{\Hd}
        		\right) \\
    	&\qquad \cdot
    		  \norm{\bh}_{H^1\cap L^{\infty}}\norm{\bet}_{\Hd}.
	\end{align*}
	Arguing similarly for the second summand and using the differentiability properties of $\rho$, we obtain
	\begin{align*}
	&
    	\abs{
    		\int_{\Omega}
    		  \lambda_n\rho^{\prime}(\btheta_n)\bh\bw_n
    		  \cdot
    		\bet\text{\,d}x
    		-\int_{\Omega}
    		  \lambda\rho^{\prime}(\btheta)\bh\bw
    		  \cdot
    		\bet\text{\,d}x
    	}
    	\\[1ex]
	&\quad\le
    	\abs{
    		\lambda_n
    		\int_{\Omega}
    		      \rho^{\prime}(\btheta_n)\bh\bw_n
    		  \cdot
    		      \bet
    		      -\rho^{\prime}(\btheta)\bh\bw_n
    		  \cdot
    		\bet\text{\,d}x
    	}
    	+\abs{
    		  \lambda_n
    		  -\lambda 
    		  \int_{\Omega}
    		          \rho^{\prime}(\btheta)\bh\bw_n
    		      \cdot
    		          \bet\text{\,d}x
    	   }
    	\\
	&\qquad
    	+\abs{
    		  \lambda\int_{\Omega}
    		      \rho^{\prime}(\btheta)\bh\bw_n
    		  \cdot
    		      \bet
    		  -\rho^{\prime}(\btheta)\bh\bw
    		      \cdot
    		  \bet\text{\,d}x
    	}
    	\\[1ex]
	&\quad\le
    	   C\Big(
    	       \norm{\btheta_n-\btheta}_{H^1\cap L^{\infty}}
    	       \norm{\bw_n}_{\Hd} 
    	       +\abs{\lambda_n-\lambda}
    	       \norm{\bw_n}_{\Hd}
    \\
	&\qquad\qquad
    	       +\norm{\bw_n-\bw}_{\Hd}
    	    \Big)
    	       \norm{\bh}_{H^1\cap L^{\infty}}
    	       \norm{\bet}_{\Hd}.
	\end{align*}
	Hence, after taking the suprema, we conclude that
	\begin{align*}
    	\norm{
    		  \partial_{\btheta}F_1(\btheta_n,\bw_n,\lambda_n)
    		  -\partial_{\btheta}F_1(\btheta,\bw,\lambda)
    	}_{\mathcal{L}\left(\HN\cap\LuN, \HD\right)}
    	\to 0,
	\end{align*}
	as $n\to\infty$.
	In summary, this implies that
	\begin{align*}
	   F:\HN\cap\LuN\times \Hd \times\R\to \HD\times \R,
	\end{align*}
	is continuously Fr\'echet differentiable.
	
	We next need to show that the partial derivative
	\begin{align*}
    	\partial_{(\bw,\lambda)}F(\bphi,\bw^{\bphi}_i,\lambda^{\bphi}_i):
    	\Hd \times \R\to \HD\times \R
	\end{align*}
	is an isomorphism.
	As $\partial_{(\bw,\lambda)}F$ is a linear operator, it suffices to show that its kernel is trivial. To this end, we consider
	\begin{align*}
        	\partial_{\bw}
        	F(\bphi,\bw^{\bphi}_i,\lambda^{\bphi}_i)\bu
    	+\partial_{\lambda}
    	   F(\bphi,\bw^{\bphi}_i,\lambda^{\bphi}_i)\mu
    	=
    	   \partial_{(\bw,\lambda)}
    	   F(\bphi,\bw^{\bphi}_i,\lambda^{\bphi}_i)(\bu,\mu)
    	=\B{0},
	\end{align*}
	which is equivalent to the equations
	\begin{align}\label{sko}
    	\BE{\bu}{\bet}
    	-\lambda^{\bphi}_i\rsp{\bu}{\bet}{\varphi}
    	-\mu\rsp{\bw^{\bphi}_i}{\bet}{\varphi}
    	=0,
	\end{align}
	for all $\bet\in \Hd$ and
	\begin{align}\label{zko}
	   2\rsp{\bu}{\bw^{\bphi}_i}{\varphi}=0.
	\end{align}
	Testing $\eqref{sko}$ with $\bet=\bw^{\bphi}_i\in \Hd$ immediately yields $\mu=0$,
	as $\bw^{\bphi}_i$ is an $\Lzrp$-normalized eigenfunction. As $\lambda_i^{\bphi}$ is assumed to be simple, we obtain from $\eqref{zko}$ that $\bu=\B{0}$. This shows that the operator is injective.
	
	To verify surjectivity we take an arbitrary tuple $(\B{f},\kappa)\in \HD\times \R$ and we need to show that there exists a solution $(\bu,\mu)\in \Hd\times \R$ of the system
	\begin{align}\label{Impsyst}
    	\begin{split}
            	\BE{\bu}{\bet}
            	-\lambda^{\bphi}_i
            	\rsp{\bu}{\bet}{\varphi}
            	-\mu\rsp{\bw^{\bphi}_i}{\bet}{\varphi}
        	&=
            	\langle\B{f},\bet \rangle_{H^{-1},H^1},\\
            	2\rsp{\bu}{\bw^{\bphi}_i}{\varphi}
        	&=\kappa,
    	\end{split}	
	\end{align}
	for all $\bet\in \Hd$. Choosing $\mu=-\langle\B{f},\bw^{\bphi}_i \rangle_{H^{-1},H^1}$, we infer that a solution $\bu\in \Hd$ needs to fulfill
	\begin{align}\label{Laximp}
        	\BE{\bu}{\bet}
        	-\lambda^{\bphi}_i
        	\rsp{\bu}{\bet}{\varphi}
    	&= 
        	\langle\B{f},\bet \rangle_{H^{-1},H^1}
        	-\langle\B{f},\bw^{\bphi}_i \rangle_{H^{-1},H^1}
        	\rsp{\bw^{\bphi}_i}{\bet}{\varphi},
	\end{align}
	for all $\bet\in \Hd$. Testing again with the normalized eigenfunction $\bw_i^{\bphi}\in \Hd$, we deduce that
	\begin{align*}
        	\langle\B{f},\bw^{\bphi}_i \rangle_{H^{-1},H^1}
        	-\langle\B{f},\bw^{\bphi}_i \rangle_{H^{-1},H^1}
        	\rsp{\bw^{\bphi}_i}{\bw^{\bphi}_i}{\varphi}
    	=
    	   0.
	\end{align*}
	Hence, Lemma~\ref{Freset} implies the existence of a function $\bu^{\perp}\in \Hd\cap \langle \bw^{\bphi}_i\rangle _{\text{span}}^{\perp,\Lzrp}$, such that any solution of $\eqref{Laximp}$ can be written as $\bu^{\perp}+\alpha \bw^{\bphi}_i\in \Hd$ and vice versa.
	Using the second equation of $\eqref{Impsyst}$, we finally conclude that
	\begin{align}\label{bij}
    	\left(
    	   \bu^{\perp}+\frac{\kappa}{2}\bw^{\bphi}_i,
    	   -\langle \B{f},\bw^{\bphi}_i \rangle 
    	\right)\in \Hd\times \R,
	\end{align}
	is a solution of $\eqref{Impsyst}$.
	
	In summary, this proves that
	\begin{align*}
    	\partial_{(\bw,\lambda)}
    	F(\bphi,\bw^{\bphi}_i,\lambda^{\bphi}_i): 
    	\Hd \times \R\to \HD\times \R
	\end{align*}
	is bijective, and thus an isomorphism.
	
	As now all requirements are verified, the implicit function theorem can be applied to the equation $F\left(\bphi,\bw^{\bphi}_i,\lambda^{\bphi}_i\right)=\B{0}$. It implies that there exists a radius $r_0>0$ such that the mapping
	\begin{align*}
        	S^{\bphi}_{i}: 
        	B_{r_0}(\bphi)\subset 
        	\HN\cap\LuN 
    	&\to
        	B_{r_i^\bphi}
        	\big(
        	   (\DB{w}{\varphi}_i,\lambda_i^{\bphi})
        	\big)
        	\subset\Hd\times \R,
	\end{align*}
	is well-defined, continuously Fr\'echet differentiable and satisfies
	\begin{align*}
		F\big(\tilde{\bphi},S^{\bphi}_i(\tilde{\bphi})\big)=\B{0},
	\end{align*}
	for all $\tilde{\bphi}\in B_{r_0}(\bphi)$. This means that
	\begin{align}\label{ewpr}
    	\begin{split}
        	   \left\langle
                    \mathcal{E}\left({S^{\bphi}_{i,1}(\tilde{\bphi})}\right)
                    \mathcal{E}\left({\bet}\right)
                \right\rangle_{\mathbb{C}(\tilde{\bphi})}    
        	&=
        	   S^{\bphi}_{i,2}(\tilde{\bphi})
        	   \int_{\Omega}
        	           \rho(\tilde{\bphi})S^{\bphi}_{i,1}(\tilde{\bphi})
        	       \cdot 
        	           \bet\text{\,d}x,\\
        	   \int_{\Omega}
        	   \rho(\tilde{\bphi})\big|S^{\bphi}_{i,1}(\tilde{\bphi})\big|^2\text{\,d}x
    	   &=1,
	\end{split}
	\end{align}
	for all $\bet\in \Hd$, and thus, $S^{\bphi}_{i,1}(\tilde{\bphi})\in \Hd$ is a $L^2_{\tilde{\bphi}}(\Omega,\R^d)$-normalized eigenfunction to the eigenvalue $S^{\bphi}_{i,2}(\tilde{\bphi})$.
	
	However, it is still not clear whether the eigenvalues $S^{\bphi}_{i,1}(\tilde{\bphi})$ and $\lambda^{\tilde{\bphi}}_i$ are actually identical.
	By construction, it holds that         $S_i^{\bphi}(\bphi)=\left(\DB{w}{\varphi}_i,\lambda_i^{\bphi}\right)$. 
	We now recall that, according to Lemma~\ref{llip} and Lemma~\ref{dwko}, both $\bphi\mapsto\lambda_i^{\bphi}$ and $\bphi\mapsto\bw_i^\varphi$ are continuous on $\HL$. In combination with \eqref{ewpr}, we conclude that there exists a radius $\delta_i^\bphi\in (0,r_0]$ such that 
	\begin{align}\label{rabs}
		  \big(\bw^{\tilde{\bphi}}_i,\lambda^{\tilde{\bphi}}_i\big)
		= S^{\bphi}_{i}(\tilde{\bphi})
		  \quad\text{for all}\; \tilde\bphi\in B_{r_i^\bphi}(\bphi).
	\end{align}
	Via restriction to the ball $B_{\delta_i^\bphi}(\bphi)$, we can thus rewrite the operator $S_i^\bphi$ as
	\begin{align*}
        	S^{\bphi}_i: 
        	B_{\delta_i^\bphi}(\bphi)
        	\subset 
        	\HN\cap\LuN
    	&\to 
        	\Hd\times \R,\\
        	\tilde{\bphi}
    	&\mapsto 
    	   \big(\bw^{\tilde{\bphi}}_i,\lambda^{\tilde{\bphi}}_i\big).
	\end{align*}
	
	It remains to compute the Fr\'echet derivative of $S_i^{\bphi}$ at the point $\bphi \in B_{\delta_i^\bphi}(\bphi)$, which means computing the desired Fr\'echet derivatives of the $i$-th eigenvalue and the corresponding eigenfunction with respect to $\bphi$.
	
	To this end, let $\bphi \in B_{\delta_i^\bphi}(\bphi)$ be arbitrary. Using the chain rule, we conclude that the Fr\'echet derivative 
	\begin{align*}
		(S_i^\bphi)'\bh =  \Big( (\bw^\bphi_i)^{\prime} \bh , (\lambda^\bphi_i)^{\prime} \bh \Big),
	\end{align*}
	satisfies the equation
	\begin{align}
	\label{EQ:FD}
    		  \partial_{(\bw,\lambda)}
    		  F(\bphi,S_i(\bphi)) 
    		  \Big( (\bw^\bphi_i)^{\prime} \bh , (\lambda^\bphi_i)^{\prime} \bh \Big) 
    		= -\partial_{\bphi}F(\bphi,S_i(\bphi))\bh
    		  \quad\text{in}\; \HD\times \R
	\end{align} 
	for any direction $\bh\in \HL$.
	With the partial derivatives computed in $\eqref{darpa}$-$\eqref{darpf}$ we obtain
	\begin{align*}
    	   \partial_{\bphi}F(\bphi,S_i(\bphi))\bh
    	=
    	   \begin{pmatrix}
    	       -\nabla\cdot
    	           \left[
    	               \C^{\prime}(\bphi)\bh\E(\bw^{\bphi}_i)
    	           \right]
    	       -\lambda^{\bphi}_i\rho^{\prime}(\bphi)\bh\bw^{\bphi}_i\\[1ex]
    	       \int_{\Omega}\rho^{\prime}(\bphi)\bh\abs{\bw^{\bphi}_i}^2\text{\,d}x
    	   \end{pmatrix}
    	\in \HD\times \R,
	\end{align*}
	for all $\bh\in \HL$, and
	\begin{align*}
    	\begin{aligned}
        	   &\partial_{(\bw,\lambda)}F\big(\bphi,S_i(\bphi)\big)(\bu,\mu)
        	=
            	\partial_{\bw}F\big(\bphi,S_i(\bphi)\big)\bu
            	+\partial_{\lambda}F\big(\bphi,S_i(\bphi)\big)\mu\\[1ex]
        	&\quad =
            	\begin{pmatrix}
            	   -\nabla\cdot \C(\bphi)\E(\bu)
            	   -\lambda^{\bphi}_i \rho(\bphi)\bu
            	   -\mu \rho(\bphi)\bw^{\bphi}_i \\[1ex]
            	   2\int_{\Omega}\rho(\bphi)\bw^{\bphi}_i\cdot\bu\text{\,d}x
            	\end{pmatrix}
        	\in \HD\times \R,
    	\end{aligned}
	\end{align*}
	for all $(\bu,\mu)\in \Hd\times \R$. 
	Consequently, \eqref{EQ:FD} is equivalent to the system \eqref{Impsyst} written for
	\begin{align*}
		      (\bu,\mu) 
		&= \Big( (\bw^\bphi_i)^{\prime} \bh , (\lambda^\bphi_i)^{\prime} \bh \Big), \\[1ex]
		      \B{f}
		&=\nabla\cdot \big[\C^{\prime}(\bphi)\bh\, \E(\bw^{\bphi}_i)\big]
			+\lambda^{\bphi}_i\rho^{\prime}(\bphi)\bh\bw^{\bphi}_i,\\[1ex]
		      \kappa 
		&= -\int_{\Omega}\rho^{\prime}(\bphi)\bh\abs{\bw^{\bphi}_i}^2\text{\,d}x. 
	\end{align*}
	Recalling the above discussion of surjectivity, we already know from \eqref{bij} that  
	$\mu=-\langle \B{f},\bw^{\bphi}_i\rangle_{H^{-1},H^1}$ 
	which directly yields
	\begin{align*}
    	   \left(\lambda^{\bphi}_i\right)^{\prime}\bh
    	=
    	   \BED{\bw^{\bphi}_i}{\bw^{\bphi}_i}
    	   -\lambda^{\bphi}_i
    	   \int_{\Omega}\rho^{\prime}(\bphi)\bh\abs{\bw^{\bphi}_i}^2\text{\,d}x.
	\end{align*}
	Plugging this into \eqref{Impsyst} with the above choices for $\bu$, $\mu$ $\B{f}$ and $\kappa$, we conclude that $(\bw^{\bphi}_i)^{\prime}\bh$ satisfies
	\begin{alignat}{2}
    	\begin{aligned}\label{chw}
        	&
            	\BE{{\left(\bw^{\bphi}_i\right)^{\prime}\bh}}{\B{\eta}}
            	-\lambda^{\bphi}_i
            	\int_{\Omega}\rho(\bphi)
            	   \left[
            	       \left(\bw^{\bphi}_i\right)^{\prime}\bh
            	   \right]
            	\cdot\B{\eta}\text{\,d}x\\
        	&\quad=	
            	-\BED{\bw^{\bphi}_i}{\B{\eta}}
            	+\lambda^{\bphi}_i
            	\int_{\Omega}\rho^{\prime}(\bphi)\bh\bw^{\bphi}_i
            	   \cdot 
            	\B{\eta}\text{\,d}x\\
        	&\qquad
            	+\left(\lambda^{\bphi}_i\right)^{\prime}\bh
            	\int_{\Omega}\rho(\bphi)\bw^{\bphi}_i
            	   \cdot
            	\B{\eta}\text{\,d}x
    	\end{aligned}
	\end{alignat}
	for all $\B{\eta}\in \Hd$, and
	\begin{align*}
		   \rsp{(\bw^{\bphi}_i)^{\prime}\bh}{\bw^{\bphi}_i}{\varphi}
        =
            \frac{\kappa}{2}
        =
            -\frac{1}{2}\int_{\Omega}\rho^{\prime}(\bphi)\bh\abs{\bw^{\bphi}_i}^2\text{\,d}x.
	\end{align*}
	This completes the proof.
\end{proof}


\section{Eigenvalue optimization} \label{SEC:OPT:1}

We can now apply the theory developed in Section~3 and Section~4 to show that the optimization problem $\eqref{Pepsla}$ (that was introduced in Subsection~\ref{SOPT}) possesses a minimizer if the set $\B{\mathcal{G}}^{\B{m}}\cap \B{U}_c$ is non-empty. Here, the assumption that the set of admissible phase-fields is non-empty is actually necessary as the sets $S_0$ and $S_1$ could be chosen in such a way that no $\bphi\in \B{U}_c$ can have the desired regularity $\HN$.

\begin{Thm}[Existence of a minimizer to \eqref{Pepsla}] \label{Exm}
	Suppose that the set $\B{\mathcal{G}}^{\B{m}}\cap \B{U}_c$ is non-empty. Then the problem $\eqref{Pepsla}$ possesses a minimizer $\overline\bphi\in \B{\mathcal{G}}^{\B{m}}\cap \B{U}_c$.
\end{Thm}

\begin{proof}
	To proof the assertion, we apply the direct method in the calculus of variations.
	Recalling that $\Psi$ is $C^1$ and bounded from below, we first observe that the objective functional is bounded by
	\begin{align*}
	   -c_{\Psi}\le J^{\eps}(\bphi)<\infty \quad \text{for all}\;\, \mathcal{F}_{\text{ad}}=\B{\mathcal{G}}^{\B{m}}\cap \B{U}_c.
	\end{align*}
	Since $\mathcal{F}_{\text{ad}}$ is non-empty, the infimum
	\begin{align*}
	   \bar J:=\underset{\bphi\in \mathcal{F}_{\text{ad}}}{\inf}
	   J^{\eps}(\bphi)
	\end{align*}
	exists in $\R$.
	Thus, there exists a minimizing sequence $\left(\bphi_k\right)_{k\in\N}\subset \mathcal{F}_{\text{ad}}$ with $J(\bphi_k)\to \bar J$ as $k\to\infty$. Using the fact that $\left(\bphi_k\right)_{k\in \mathbb{N}}$ is bounded in $\HN$, we infer that
	\begin{align*}
	   \bphi_k&\rightharpoonup\oB{\varphi}\quad \text{in } \HN,\quad \text{as $k\to \infty$},
	\end{align*}
	along a non-relabeled subsequence.
	As the sequence $\left(\bphi_k\right)_{k\in\N}$ lies in $\B{\mathcal{G}}^{\B{m}}$ it is also bounded in $\LuN$. Hence, Theorem~\ref{slw} implies that 
	\begin{align*}
	   \lambda_{i_j}^{\bphi_k}\to \lambda_{i_j}^{\bphi},\quad \text{as $k\to \infty$},
	\end{align*}
	for all $j=1,\dots, l$.
	As $\Psi$ is continuous and the Ginzburg--Landau energy is weakly lower semi-continuous, we conclude that
	\begin{align*}
    	   J^{\eps}_l(\oB{\varphi})
    	\le 
    	   \underset{k\to\infty}{\lim\inf}J^{\eps}_l(\bphi_k) = \bar J.
	\end{align*}
	This directly implies that $J^{\eps}_l(\oB{\varphi}) = \bar J$ and thus, $\overline\bphi$ is a minimizer of the functional $J$ on the set $\mathcal{F}_{\text{ad}}$. This completes the proof.
	
\end{proof}

Now, invoking the differentiability properties established in Section~5, we can derive a first-order necessary condition for local optimality.

\begin{Thm}[The optimality system to \eqref{Pepsla}] \label{VU2}
	Let $\bphi\in \left(\B{\mathcal{G}}^{\B{m}}\cap \B{U}_c\right)$ be a local minimizer of the optimization problem $\eqref{Pepsla}$, i.e., there exists $\delta>0$ such that 
	\begin{align*}
	   J_l^{\varepsilon}(\btheta) \ge J_{l}^{\varepsilon}(\bphi)\quad 
	   \text{for all $\btheta\in \B{\mathcal{G}}^{\B{m}}\cap \B{U}_c$ with } \norm{\btheta-\bphi}_{\HL}<\delta.  
	\end{align*}
	Suppose that the eigenvalues $\lambda^{\bphi}_{i_1},\dots,\lambda^{\bphi}_{i_l}$ are simple and let us fix $\Lzrp$-normalized eigenfunctions $\bw^{\bphi}_{i_1},\dots,\bw^{\bphi}_{i_l}\in \Hd$ to the eigenvalues $\lambda^{\bphi}_{i_1},\dots,\lambda^{\bphi}_{i_l}$, respectively. 
	
	Then the following optimality system is satisfied:
	\begin{itemize}
    		\item The state equations
    		  \begin{align}\tag{${SE}_j$}
    		      \begin{cases}
    		          \begin{array}{rll}
    		              -\nabla\cdot\left[\C(\bphi)\mathcal{E}(\bw_{i_j}^{\bphi})\right]
    		          &=
    		              \lambda^{\bphi}_{i_j}
    		              \rho(\bphi)\bw_{i_j}^{\bphi}
    		          &\quad
    		              \text{in }\Omega,\\
    		              \bw_{i_j}^{\bphi}
    		          &=
    		              \B{0}
    		          &\quad
    		              \text{on }\Gamma_D,\\
    		              \left[
    		                  \C(\bphi)\mathcal{E}(\bw_{i_j}^{\bphi})
    		              \right]\B{n}
    		          &=
    		              \B{0}
    		          &\quad
    		              \text{on }\Gamma_0,
    		      \end{array}
    		  \end{cases}
    		\end{align}
    		are satisfied for all $j\in\{1,\dots,l\}$. 
    		\item 
    		The variational inequality
        		\begin{align}
        		\tag{${VI}$}
            		\begin{aligned}
                		  0
                		&\le 
                    		\gamma\eps
                    		\int_{\Omega}\nabla\bphi:\nabla(\btheta-\bphi)\,\text{d}x
                    		+\frac{\gamma}{\eps}
                    		\int_{\Omega}\psi_0^{\prime}(\bphi)(\btheta-\bphi)\text{\,d}x \\[1ex]
                		&\qquad 
                    		+\sum_{j=1}^{l}
                    		\Bigg\{
                    		      \Psi'_{\lambda_{i_j}}
                    		          \big(\lambda^{\bphi}_{i_1},\dots,\lambda^{\bphi}_{i_l}\big)\;
                    		              \Bigg(
                    		                  \BEDp{\bw^{\bphi}_{i_j}}{\bw^{\bphi}_{i_j}} \\[-1ex]
                		&\qquad\qquad\qquad\qquad\qquad\qquad\qquad
                    		                  -\lambda^{\bphi}_{i_j}
                    	                       	\int_{\Omega}
                    		                          \rho^{\prime}(\bphi)\left(\btheta
                    		                          -\bphi\right)\abs{\bw^{\bphi}_{i_j}}^2
                    		                      \text{\,d}x
                    		                  \Bigg)
                    		  \Bigg\}
            		\end{aligned}
        		\end{align}
    		is satisfied for all $\btheta\in \left(\B{\mathcal{G}}^{\B{m}}\cap \B{U}_c\right)$ and all $j\in\{1,\dots,l\}$.
	\end{itemize}
\end{Thm}

\begin{proof}
	Since $\B{\mathcal{G}}^{\B{m}}\cap \B{U}_c$ is convex, it holds that $\bphi+t(\btheta-\bphi) \in \B{\mathcal{G}}^{\B{m}}\cap \B{U}_c$ for all $\btheta\in\B{\mathcal{G}}^{\B{m}}\cap \B{U}_c$ and all $t\in[0,1]$. As the objective functional $J_l^{\varepsilon}$ is Fr\'echet differentiable by chain rule, we know that
	\begin{align*}
		0\le \frac{\mathrm d}{\mathrm d t} J\big(\bphi+t(\btheta-\bphi) \big)\big\vert_{t=0} = J'\big(\bphi\big)(\btheta-\bphi).
	\end{align*}
	Using \eqref{lhdef} in Theorem~\ref{difewv} it is now straightforward to check that $J'\big(\bphi\big)(\btheta-\bphi)$ is identical with the right-hand side of the variational inequality. This completes the proof. 
\end{proof}
\begin{Rem}
	\normalfont
	If the first eigenvalue $\lambda_1^\bphi$ is not simple but only $\lambda_1^\bphi$ and further simple eigenvalues appear in the objective functional, we can still derive a variational inequality by means of the semi-differentiability established in Theorem~\ref{GDO}. This is because in the above proof only variations $\bphi+t(\btheta-\bphi)$ with positive $t$ are considered.
	
	To be precise, let us assume that the multiplicity of the eigenvalue $\lambda_1^\bphi$ is $M\in\N$. This means that
    	\begin{align*}
    		\lambda_1^\bphi = \lambda_2^\bphi = ... = \lambda_M^\bphi.
    	\end{align*}
    	If now $\lambda_1^{\bphi}$ appears in \eqref{Pepsla} but none of the eigenvalues $\lambda_2^\bphi = ... = \lambda_M^\bphi$ does, the term
    	\begin{align*}
        	   \BEDp{\bw^{\bphi}_{1}}{\bw^{\bphi}_{1}}
        	-\lambda^{\bphi}_{1}
        	   \int_{\Omega}
        	   \rho^{\prime}(\bphi)\left(\tilde{\bphi}
        	-\bphi\right)\big| \bw^{\bphi}_{1} \big|^2,
    	\end{align*}
    	in the variational inequality has to be replaced by 
    	\begin{align*}
        	\inf\left\{
        		\BEDp{\bu}{\bu}
        		-\lambda^{\bphi}_{1}
        		\left(\bu,\bu\right)_{
        			\rho^{\prime}(\bphi)\left(
        			\tilde{\bphi}
        			-\bphi
        		\right)
        		}
    	       \left|\;
                	\begin{aligned}
                		&\bu\in \Hd \text{ is an} \\
                		&\text{eigenfunction to }\lambda_{1}^\bphi \\ 
                		&\text{with } \norm{\bu}_{\Lzrp}=1 
                	\end{aligned} 
    	       \right.\right\}.
    	\end{align*}
    	Of course, if $\ev$ is simple (i.e., $M=1$) both terms coincide.
	
	In the following we will only discuss the case of simple eigenvalues, but keep the fact in mind that it is not necessary to require simplicity of the first eigenvalue.
\end{Rem}


\section{Combination of compliance and eigenvalue optimization} \label{SEC:OPT:2}

We now want to analyze the optimization problem \eqref{Kepsl} (that was introduced in Subsection~\ref{COPT}) by establishing results similar to those in Section~\ref{SEC:OPT:1}
To this end, we will use the control-to-state operator
\begin{align*}
    S:\HL\to \HC, \quad \bphi\mapsto \bu(\bphi)
\end{align*}
that was introduced in \cite{Blank} and maps any $\bphi\in \HL$ onto its corresponding solution $\bu=\bu(\bphi)$ of the state equation \eqref{wState}.
This allows us to consider the reduced optimization problem
\begin{align}\tag{$\mathcal{K}^{\eps*}_{l}$}\label{Kepsl*}
    \left\{
        \begin{aligned}
                &\min &&I_l^{\eps}(\bphi)
            =
                \alpha F\big(S(\bphi),\bphi\big)
                +\beta J_0\big(S(\bphi),\bphi\big)
                +\gamma E^{\eps}(\bphi)
                +\Psi(\lambda^{\bphi}_{i_1},\dots,\lambda^{\bphi}_{i_l})\\
           &\text{ s.t.}&& 
                \bphi\in \mathcal{\B{\mathcal{G}}}^{\B{m}}\cap \B{U}_c,
                \eqref{wState} \text{ is fulfilled},\\
          &&&\text{and } 
                \lambda^{\bphi}_{i_1},\dots,\lambda^{\bphi}_{i_l} 
                \text{ are}\text{ eigenvalues of } \eqref{WEstate}
        \end{aligned}
    \right.
\end{align}
(with $\alpha,\beta\ge 0$, $\gamma,\eps>0$ and $\B{m}\in (0,1)^N\cap\Sigma^N$), which is obviously equivalent to the original problem \eqref{Kepsl}.

The following theorem ensures the existence of a minimizer to \eqref{Kepsl*} or \eqref{Kepsl}, respectively.

\begin{Thm}[Existence of a minimizer to \eqref{Kepsl*}]
	Suppose that the set $\mathcal{\B{\mathcal{G}}}^{\B{m}}\cap \B{U}_c$ is non-empty. Then the problem $\eqref{Kepsl*}$ has a minimizer $\overline\bphi\in \mathcal{\B{\mathcal{G}}}^{\B{m}}\cap \B{U}_c$.
\end{Thm}
\begin{proof}
	The assertion can be verified by simply combining the proof in \cite[Thm~4.1]{Blank} and the proof of Theorem~\ref{Exm}. Therefore we omit the details.
\end{proof}
From the differentiability properties deduced in \cite{Blank} and in this paper we obtain a variational inequality for the problem $\eqref{Kepsl}$. Note that for $\nu\in (0,1)$, the functional $J_0$ is in general not differentiable where the integral raised to the power $\nu$ is equal to zero. Hence, as in \cite{Blank}, we only consider $(\bu,\bphi)$ such that
\begin{align*}
    \int_{\Omega}c\big(1-\varphi^N\big)\abs{\bu-\bu_{\Omega}}^2\text{\,d}x\neq 0,
\end{align*}
if $\beta\neq 0$. 

Eventually, we can state the optimality system for the combined problem where the first-order necessary condition for local optimality is incorporated.
\begin{Thm}[The optimality system to \eqref{Kepsl*}]
	Let $\bphi\in \left(\B{\mathcal{G}}^{\B{m}}\cap \B{U}_c\right)$ be a local minimizer of the optimization problem $\eqref{Kepsl*}$, i.e., there exists $\delta>0$ such that 
	\begin{align*}
    	   I_l^{\varepsilon}(\btheta) \ge I_l^{\varepsilon}(\bphi)\quad 
    	\text{for all $\btheta\in \B{\mathcal{G}}^{\B{m}}\cap \B{U}_c$ with }      \norm{\btheta-\bphi}_{\HL}<\delta.  
	\end{align*}
	Suppose that the eigenvalues $\lambda^{\bphi}_{i_1},\dots,\lambda^{\bphi}_{i_l}$ are simple and let us fix $\Lzrp$-normalized eigenfunctions $\bw^{\bphi}_{i_1},\dots,\bw^{\bphi}_{i_l}\in \Hd$ to the eigenvalues $\lambda^{\bphi}_{i_1},\dots,\lambda^{\bphi}_{i_l}$, respectively.  
	
	Then, there exist a state $\bu\in \HC$ and an adjoint state $\B{p}\in \HC$ such that the tuple
	\begin{align*}
    	&\Big(
        	\bu,
        	\bphi,
        	\B{p},
        	\big(\DB{w}{\varphi}_{i_j}\big)_{j=1}^{l},
        	\big(\lambda_{i_j}\big)_{j=1}^{l}
        \Big)\\
    	&\quad \in 
        	\HC
        	\times\big(\B{\mathcal{G}^m\cap \B{U}_c}\big)
        	\times \HC
        	\times\big(\Hd\big)^l 
        	\times \R^l
	\end{align*}
	fulfills the following optimality system:
	\begin{itemize}
	\item The state equations
	\begin{align}\tag{$SE^*$}\label{SE}
    	\begin{cases}
        	\begin{array}{rll}
            	   -\nabla\cdot\left[\C(\bphi)\mathcal{E}(\bu)\right]
            	&=
            	   \big(1-\varphi^N\big)\B{f}
            	& \quad\text{in }\Omega,\\
            	   \bu
            	&=
            	   \B{0}
            	&\quad\text{on }\Gamma_C,\\
            	   \left[\C(\bphi)\mathcal{E}(\bu)\right]\B{n}
            	&=
            	   \B{g}
            	&\quad\text{on }\Gamma_g,
        	\end{array}
    	\end{cases}
	\end{align}
	and
	\begin{align}\tag{${SE}_j^*$}
    	\begin{cases}
        	\begin{array}{rll}
        	   -\nabla\cdot\left[\C(\bphi)\mathcal{E}(\bw_{i_j}^{\bphi})\right]
        	&=
        	   \lambda^{\bphi}_{i_j}\rho(\bphi)\bw_{i_j}^{\bphi}
        	& \quad\text{in }\Omega,\\
        	   \bw_{i_j}^{\bphi}
        	&=
        	   \B{0}
        	&\quad\text{on }\Gamma_D,\\
        	   \left[\C(\bphi)\mathcal{E}(\bw_{i_j}^{\bphi})\right]\B{n}
        	&=\B{0}
        	&\quad\text{on }\Gamma_0,
        	\end{array}
    	\end{cases}
	\end{align}
	for $j=1,\dots,l$, are satisfied in the weak sense.
	\item The adjoint equation
	\begin{align}\tag{$AE^*$}\label{AE}
    	\begin{cases}
        	\begin{aligned}
            	   -\nabla\cdot\left[\C(\bphi)\mathcal{E}(\B{p})\right]
            	&=
            	   \alpha\big(1-\varphi^N\big)\B{f}\\
            	&\phantom{=}
            	   +2\beta
            	   \nu J_0(\bu,\bphi)^{\frac{\nu-1}{\nu}}
            	   c\big(1-\varphi^N\big)\left(\bu-\bu_{\Omega}\right)
            	&&\text{in }\Omega,\\
            	   \B{p}
            	&=
            	   \B{0}
            	&&\text{on }\Gamma_C,\\
            	   \left[\C(\bphi)\mathcal{E}(\B{p})\right]\B{n}
            	&=
            	   \alpha\B{g}
            	&&\text{on }\Gamma_g,
        	\end{aligned}
    	\end{cases}
	\end{align} 
	is satisfied in the weak sense.
	\item The variational inequality
	\begin{align}\tag{$VI^*$}
    	\begin{aligned}
           	0
            	\,\le\, 
        	&\gamma\eps 
            	\int_{\Omega}
            	   \nabla\bphi:\nabla (\btheta-\bphi)
            	\textup{\,d}x
            +\frac{\gamma}{\eps}
	            \int_{\Omega}
	            \psi_0^{\prime}(\bphi)\cdot(\btheta-\bphi)
	            \textup{\,d}x\\[0.5ex]
        	&-\beta\nu 	J_0(\bu,\bphi)^{\frac{\nu-1}{\nu}}
            	   \int_{\Omega}
            	       c\big(\vartheta^N-{\varphi}^N\big)\left|\bu-\bu_{\Omega}\right|^2
            	   \textup{\,d}x\\[0.5ex]
        	&-\int_{\Omega}
            	   \big(\vartheta^N-{\varphi}^N\big)\B{f}\cdot(\alpha\bu+\B{p})
            	\textup{\,d}x
            	\;-\langle 
            	   \mathcal{E}(\B{p}),\mathcal{E}(\bu)
            	\rangle_{\C^{\prime}(\bphi)(\btheta-\bphi)}\\[0.5ex]
        	&+\sum_{j=1}^{l}
            	   \left\{
            	       \Psi_{\prime \lambda_{i_j}}
            	       \big(\lambda^{\bphi}_{i_1},\dots,\lambda^{\bphi}_{i_l}\big)
            	           \left(
            	               \BEDp{\bw^{\bphi}_{i_j}}{\bw^{\bphi}_{i_j}} \right.\right.\\[0.5ex]
        	&\qquad\qquad\qquad\qquad\qquad\qquad
            	\left.\left. -\lambda^{\bphi}_{i_j}
            	               \int_{\Omega}
            	                   \rho^{\prime}(\bphi)\left(\btheta-\bphi\right)\big|\bw^{\bphi}_{i_j}\big|^2
            	               \textup{\,d}x
            	           \right)
            	   \right\},\\[1ex]
    	\end{aligned}
	\end{align}
	is satisfied for all $\btheta\in \B{\mathcal{G}}^{\B{m}}\cap \B{U}_c$.
	\end{itemize}
\end{Thm}

\begin{proof}
	Using the properties of the control-to-state operator $S$, the assertion can be proved proceeding similarly as in the proof of Theorem~\ref{VU2}.
\end{proof}

\section*{Acknowledgment}
The authors were supported by the RTG 2339 ``Interfaces, Complex Structures, and Singular Limits''
of the German Science Foundation (DFG). 
The support is gratefully acknowledged.


\footnotesize\setlength{\parskip}{0cm}
\bibliography{OptEig}

\end{document}